\documentclass[12pt]{amsart}

\usepackage{amstext}
\usepackage{latexsym}
\usepackage{a4wide}
\usepackage{amssymb,amsbsy}
\usepackage{amsmath,amsthm,amscd}
\usepackage{enumitem}
\usepackage[pdftex, bookmarks=true, linkbordercolor={0 1 0}]{hyperref}

\newcommand{\R}{\mathbb R}
\newcommand{\Q}{\mathbb Q}
\newcommand{\C}{\mathbb C}
\newcommand{\Z}{\mathbb Z}
\newcommand{\N}{\mathbb N}
\newcommand{\A}{\mathbb A}
\newcommand{\F}{\mathbb F}

\newcommand{\eps}{\varepsilon}
\newcommand{\minus}{\backslash}

\newcommand{\Hom}{\mathop{\rm Hom}\nolimits}

\newcommand{\tr}{\mathop{\rm tr}\nolimits}
\newcommand{\vol}{\mathop{\rm vol}\nolimits}
\newcommand{\GL}{\mathop{\rm GL}\nolimits}
\newcommand{\SO}{\mathop{{\rm SO}}\nolimits}

\newcommand{\Orth}{\mathop{{\rm O}}\nolimits}
\newcommand{\Unit}{\mathop{{\rm U}}\nolimits}
\newcommand{\SL}{\mathop{\rm SL}\nolimits}
\newcommand{\PSL}{\mathop{\rm PSL}\nolimits}
\newcommand{\SU}{\mathop{\rm SU}\nolimits}

\newcommand{\id}{\mathop{\rm id}\nolimits}

\newcommand{\res}{\mathop{\rm res}}
\newcommand{\Ind}{\rm Ind}
\newcommand{\diag}{\mathop{\rm diag}}

\newcommand{\supp}{\mathop{\rm supp}}
\newcommand{\Mat}{\mathop{\rm Mat}\nolimits}

\newcommand{\cpt}{{\bf K}}

\newcommand{\Ad}{\mathop{{\rm Ad}}\nolimits}

\newcommand{\aaa}{\mathfrak{a}}

\newcommand{\ooo}{\mathfrak{o}}

\newcommand{\qqq}{\mathfrak{q}}

\newcommand{\AAA}{\mathcal{A}}
\newcommand{\HHH}{\mathcal{H}}
\newcommand{\OOO}{\mathcal{O}}

\newcommand{\FFF}{\mathcal{F}}

\newcommand{\NNN}{\mathcal{N}}

\newcommand{\EEE}{\mathcal{E}}

\newcommand{\One}{\mathbf{1}}

\newcommand{\signature}{\underline{r}}
\newcommand{\Jgeom}{J_{\mathrm{geom}}}

\renewcommand{\Re}{{\rm Re}}

 \theoremstyle{plain}
 \newtheorem{theorem}{Theorem}
 
 \newtheorem{lemma}[theorem]{Lemma}
 \newtheorem{cor}[theorem]{Corollary}
 \newtheorem{proposition}[theorem]{Proposition}
 
 \newtheorem{rem}[theorem]{Remark}
 
 \theoremstyle{definition}

 \makeatletter
\def\namedlabel#1#2{\begingroup
   \def\@currentlabel{#2}
   \label{#1}\endgroup
}
\makeatother

\begin{document}

\title[Limit multiplicities for $\SL_2(\R^{\lowercase{r}_1}\oplus\C^{\lowercase{r}_2})$]{Limit multiplicities  for  $\SL_2(\OOO_F)$ in $\SL_2(\R^{\lowercase{r}_1}\oplus\C^{\lowercase{r}_2})$}
\author{Jasmin Matz}
\thanks{The majority of this work was done during a visit at the MSRI in 2014 which was supported by the NSF under grant No. DMS-0932078. This research was partially supported by the Israel Science Foundation (grant no.\ 1676/17)}
\address{Einstein Institute of Mathematics, Edmond J. Safra Campus, The Hebrew University of Jerusalem, Givat Ram, 9190401, Israel}
\email{jasmin.matz@mail.huji.ac.il}
\begin{abstract}
We prove that the family of lattices $\SL_2(\OOO_F)$, $F$ running over number fields with fixed archimedean signature $(r_1, r_2)$,  in $\SL_2(\R^{\lowercase{r}_1}\oplus\C^{\lowercase{r}_2})$ has the limit multiplicity property. 
\end{abstract}
\maketitle

\tableofcontents

Suppose $H$ is a semisimple Lie group and $\{\Gamma_j\}_{j\in\N}$ a family of lattices in $H$, that is, each $\Gamma_j$ is a discrete subgroup of $H$ and the quotient $\Gamma_j\backslash H$ has finite measure. Let $\widehat{H}$ denote the unitary dual of $H$. This is a topological space with respect to the Fell topology, and it carries a natural measure $\mu_{{\rm Pl}}$, the Plancherel measure. $H$ acts on $L^2(\Gamma_j\backslash H)$ via the right regular representation, and we write $L^2_{\text{disc}}(\Gamma_j\backslash H)$ for the maximal subspace of $L^2(\Gamma_j\backslash H)$ such that the restriction of the right regular representation to this subspace decomposes discretely. 
Each lattice $\Gamma_j$ defines a measure $\mu_j$ on $\widehat{H}$ by setting
\[
 \mu_j=\frac{1}{\vol(\Gamma_j\backslash H)} \sum_{\pi\in \widehat{H}} m_{\Gamma_j}(\pi) \delta_\pi
\]
where $m_{\Gamma_j}(\pi)$ denotes the multiplicity of $\pi$ in $L^2_{\text{disc}}(\Gamma_j\backslash H)$, and $\delta_\pi$ is the Dirac measure on $\widehat{H}$ at $\pi$. 
If we have $\vol(\Gamma_j\backslash H)\xrightarrow{\;j\rightarrow\infty\;}\infty$, a natural question is whether 
\[
\mu_j\longrightarrow \mu_{\rm Pl}
\]
as $j\rightarrow\infty$.
This is the limit multiplicity problem, and if  $\mu_j\longrightarrow \mu_{\rm Pl}$ holds, we say  that $\{\Gamma_j\}_{j\in\N}$ has the limit multiplicity property (see below for a more precise definition).

When the $\{\Gamma_j\}_{j\in\N}$ constitute a tower of congruence subgroups, and they are uniform in $H$ (i.e., co--compact) or of rank $1$, the limit multiplicity problem is known to have a positive solution in many cases, see \cite{DGWa78,De86,Sa89,Wa90,DeHo99}. More general situations of uniform lattices have been considered in \cite{AbBe12}.  
For non--uniform lattices of higher rank much less is known. Recently, it was shown in \cite{FiLaMu_limitmult,FiLa15} that the collection of congruence subgroups of $\SL_n(\OOO_F)$ has the limit multiplicity property in  $\SL_n(F\otimes \R)$ where $F$ is a fixed number field with ring of integers $\OOO_F$. A similar statement is true for other situations arising from classical groups under some natural hypotheses (which are expected to hold unconditionally) \cite{FiLaMu_limitmult,FiLa15}.  
Note that an important common feature of all these situations is that the lattices are pairwise commensurable.
The first instances of a sequence of non--commensurable, non--uniform lattices were first studied in~\cite{Ra13,Fr16} under certain assumptions when $H=\PSL_2(\C)$ or $\PSL_2(\R)$.
For further results and a more detailed history of the topic we refer to \cite[\S1]{FiLaMu_limitmult} and \cite[\S1]{AbBe12}. 

In general, one expects that any sequence $\{\Gamma_j\}_{j\in\N}$ of congruence subgroups with $\vol(\Gamma_j\backslash H)\rightarrow\infty$ satisfies the limit multiplicity property. If the sequence consists of non--congruence lattices, this does not necessarily need to be the case, see \cite{PhSa91}.

A number of related problems are considered in \cite{ShTe16,KST16}. For example, instead of the multiplicities of representations in $L_{\text{disc}}^2(\Gamma_j\backslash H)$ as $j\rightarrow\infty$, the limit behavior of multiplicities of discrete automorphic representations  with local constraints is studied when those local properties vary. 
 
The purpose of this paper is  to consider families of non--uniform non--commensurable lattices in certain Lie groups  which arise naturally  from number fields. This extends certain results of \cite{Ra13}.
 
Before stating our results we give a precise definition of the limit multiplicity property. We use the same definition as in \cite{FiLaMu_limitmult}. 
Let $\mathbf{\Gamma}=\{\Gamma\}$ be a collection of lattices in $H$. Then $\mathbf{\Gamma}$ is said to have the \emph{limit multiplicity property} if for every $\eps>0$ the following two properties hold:
\begin{enumerate}[label=(\Roman{*})]
\item\label{LMP1} For every bounded set $A\subseteq \widehat{H}_{\text{temp}}$ with $\mu_{\text{Pl}}(\partial A)=0$ we have $|\mu_{\Gamma}(A)-\mu_{\text{Pl}}(A)|<\eps$ for all but finitely many $\Gamma\in\mathbf{\Gamma}$. Here $\widehat{H}_{\text{temp}}$ denotes the tempered part of $\widehat{H}$.
\item\label{LMP2} For every bounded set $A\subseteq \widehat{H}\minus \widehat{H}_{\text{temp}}$ we have $\mu_{\Gamma}(A)<\eps$ for all but finitely many $\Gamma\in\mathbf{\Gamma}$.
\end{enumerate}
Here a subset $A\subseteq \widehat{H}$ is called bounded if the set of infinitesimal characters of the elements in $A$ is bounded. This is equivalent to $A$ being relatively quasi-compact in $\widehat{H}$. Note that the tempered part $\widehat{H}_{\text{temp}}$ of $\widehat{H}$ equals the support of the Plancherel measure.

We are going to be concerned mainly with groups $H$ which are essentially direct products of finitely many copies of $\SL_2(\R)$ and $\SL_2(\C)$. We therefore briefly recall what the unitary duals look like for $\SL_2(\R)$ and $\SL_2(\C)$.  For $\SL_2(\C)$, the unitary dual can be identified with a disjoint union of intervals in $\C$  \cite[\S II.4]{Knapp} and its topology is the usual one inherited from $\C$.  In the case of $\SL_2(\R)$, the unitary dual has a similar description \cite{Knapp}: After removing a finite number of (non--Hausdorff) points from $\widehat{\SL_2(\R)}$, it can be identified with a disjoint union of finitely many intervals and an infinite discrete collection of points in $\C$. The topology on this part of $\widehat{\SL_2(\R)}$ is again the one inherited from $\C$, and around the removed points a basis of neighborhoods can be described explicitly \cite[\S7.6]{Folland}.

 We now describe the results of this paper. 
From now on fix $d\geq 2$, and an archimedean signature $\signature=(r_1, r_2)$ of number fields of degree $d$ over $\Q$, that is, $r_1$ denotes the number of real and $r_2$ the number of pairs of complex embeddings so that $r_1+2r_2=d$. Let $\F_{\signature}$ be the set of all number fields with archimedean signature $\signature$.
For every $F\in\F_{\signature}$ we fix an isomorphism 
\[
F_{\infty}=F\otimes\R\simeq \R^{r_1}\times \C^{r_2}=:\R^{\signature}
\]
(as $\R$-algebras), and identify all $F_{\infty}$ with each other via these isomorphisms. This also fixes embeddings $F\hookrightarrow\R^{\signature}$. Let $\OOO_F\subseteq F$ be the ring of integers in $F$
Then
\[
 \Gamma_F:=\{\pm1\}\backslash\SL_2(\OOO_F)\subseteq H:=\{\pm1\}\backslash\SL_2(\R^{\signature})
\]
defines a family of non--commensurable non--cocompact lattices in $H$ with $\vol(\Gamma_F\backslash H)\longrightarrow\infty$ as $F$ varies over $\F_{\signature}$\footnote{Here and in the rest of the paper we mean the following when saying that for a function $\phi:\F_{\signature} \longrightarrow\C$ and $\phi_\infty\in\C$ we have  $\phi(F)\longrightarrow \phi_\infty$ as $F$ varies over $F\in \F_{\signature}$:  For any $\eps>0$ there exists a finite set $A\subseteq \F_{\signature}$ such that $|\phi(F)-\phi_\infty|<\eps$ whenever $F\in\F_{\signature}\smallsetminus A$. We also write $\lim_{F\in\F_{\signature}} \phi(F)=\phi_\infty$ for short. This is equivalent to considering $\phi:\F_{\signature}\longrightarrow \C$ as a net (with $\F_{\signature}$ being a directed set via $F_1\le F_2$ iff $D_{F_1}\le D_{F_2}$ for $D_{F_i}$ the absolute discriminant of $F_i$) and taking the limit of this net.
}.

Our first result is the following:

\begin{theorem}\label{thm:main}
The family of lattices $\{\Gamma_F\}_{F\in\F_{\signature}}$  in $H$ has the limit multiplicity property.
\end{theorem}

In \cite{Ra13} the case of $\signature=(0,1)$ is considered and it is shown that the sequence of hyperbolic $3$--manifolds $\Gamma_F\backslash \SL_2(\C)/\SU(2)$, $F\in\F_{(0,1)}$, is Benjamini--Schramm-convergent\footnote{Roughly speaking, a sequence of locally Riemannian symmetric spaces $Y_j=\Gamma_j\backslash X$ (with normalization $\vol(Y_j)=1$ for any $j$) is Benjamini--Schramm convergent to $X$, if for any $r>0$, the measure of the set of all $y\in Y_j$ having injectivity radius at least $r$ tends to $1$ as $j\rightarrow 1$. See \cite[\S3]{AbBe12} for more details.} to the universal cover $\SL_2(\C)/\SU(2)$. This implies the limit multiplicity property for $\{\Gamma_F\}_{F\in\F_{(0,1)}}$. Similar results for more general sequences of lattices in $\PSL_2(\R)$ or $\PSL_2(\C)$ were established in \cite{Fr16}.

To start the proof of Theorem \ref{thm:main} we first use a variant of Sauvageot's density principle similarly as in \cite{Sa97,Sh12,FiLaMu_limitmult} (see \cite[\S2]{FiLaMu_limitmult} for a precise statement).
Let $\HHH(H)$ denote the algebra of all smooth compactly supported functions on $H$ which are left and right $\cpt_{\infty}'$-finite, where $\cpt_{\infty}'=\cpt_\infty/\{\pm 1\}\subseteq H$ is a certain maximal compact subgroup of $H$ and $\cpt_\infty$ is defined in Section~\ref{sec:contr:elements} below.
Then to prove Theorem \ref{thm:main} we need to show that for every $h\in\HHH(H)$ we have
\begin{equation}\label{eq:sauvageot:criterion}
 \mu_{\Gamma_F}(\hat{h})\longrightarrow \mu_{\text{Pl}}(\hat{h}) = h(1) 
\end{equation}
as $F$ varies over $\F_{\signature}$. Here $\hat{h}(\pi)=\tr h(\pi)$ for $\pi\in\widehat{H}$.

In fact, we are going to show the following effective version of \eqref{eq:sauvageot:criterion}:

\begin{theorem}\label{thm:effective}
 Let $\eps>0$. For every $h\in \HHH(H)$ and $F\in\F_{\signature}$ we have 
 \[
  \left|\mu_{\Gamma_F}(\widehat{h}) - \mu_{\text{Pl}}(\widehat{h})\right|
  \ll_{d, h,\eps} \vol(\Gamma_F\backslash H)^{-\frac{1}{3}+\eps}.
 \]
\end{theorem}
Here and in the following the notation $\ll_{a,b,\ldots}$ means that the implied constant depends on the quantities $a,b,\ldots$ but on no other.

\begin{rem}
A combination of our methods with \cite{FiLaMu_limitmult,FiLa15} should also give the limit multiplicity property for the bigger collection of all discrete subgroups in $H$ which are congruence subgroups of one of the $\Gamma_F$, $F\in\F_{\signature}$.
\end{rem}

We end the introduction by giving a sketch of the proof of \eqref{eq:sauvageot:criterion} and therefore Theorem~\ref{thm:main}. To prove Theorem \ref{thm:effective} each step in the proof has to be made effective. This comes out of our proof quite naturally. We want to use Arthur's trace formula so that we first translate the problem into adelic terms: Let $\A_F$ denote the ring of adeles of $F$, and $\A_{F,f}$ the finite part of $\A_F$. Let $G=\SL(2)$ as an algebraic group over $F$, and let $\cpt_f^F=\prod_{v<\infty}\cpt_v^F\subseteq G(\A_{F,f})$ be the usual maximal compact subgroup, that is, $\cpt_v^F=G(\OOO_{F_v})$ for $\OOO_{F_v}$ the ring of integers in the local field $F_v$. 
Then
\[
\Gamma_F\backslash H
\simeq G(F)\backslash G(\A_F)/\cpt_f^F,
\]
and the measure $\mu_{\Gamma_F}$ on $\widehat{H}$ becomes
\[
\mu_{\Gamma_F}= \frac{1}{\nu_F} \sum_{\pi\in\widehat{G(\A_F)}}
\dim\Hom_{G(\A_F)} \left(\pi, L^2(G(F)\backslash G(\A_F)/\cpt_f^F)\right)  ~\delta_{\pi_{\infty}},
\]
where
\begin{equation}\label{eq:def:volume}
\nu_F
= \vol(\Gamma_F\backslash H)
=\vol(G(F)\backslash G(\A_F)/ \cpt_f^F)
= \frac{\vol(G(F)\backslash G(\A_F))}{\vol(\cpt_f^F)}.
\end{equation}
Let $L^2_{\text{disc}}(G(F)\backslash G(\A_F))$ denote the discrete part of $L^2(G(F)\backslash G(\A_F))$, that is, the maximal subspace of $L^2(G(F)\backslash G(\A_F))$ which decomposes discretely under the right regular representation $R$. Equivalently, $L^2_{\text{disc}}(G(F)\backslash G(\A_F))$ equals the direct sum (with appropriate multiplicities) of those $\pi$ with $\Hom_{G(\A_F)} \left(\pi, L^2(G(F)\backslash G(\A_F)/\cpt_f^F)\right)\neq 0$.

We define $\HHH(G(\R^{\signature}))$ analogously as $\HHH(H)$ as the space of all left and right $\cpt_\infty$--finite functions in $C_c^\infty(G(\R^{\signature}))$. 
Let $h\in \HHH(G(\R^{\signature}))$ and set 
\[
h_1(g)=h(g)+ h(-g)
\]
so that $h_1\in\HHH(H)$. Let $\One_{\cpt_f^F}:G(\A_{F,f})\longrightarrow \C$ denote  the characteristic function of $\cpt_f^F$, and let $J_{\text{disc}}^F(h\cdot \One_{\cpt_f^F})$ denote the trace of $R(h\cdot \One_{\cpt_f^F})$ restricted to $L^2_{\text{disc}}(G(F)\backslash G(\A_F)^1)$. 
Then 
\begin{equation}\label{eq:sauv}
\mu_{\Gamma_F}(\hat{h_1}) = \frac{1}{\vol(G(F)\backslash G(\A_F))} J_{\text{disc}}^F(h\cdot \One_{\cpt_f^F})
\end{equation}
and $J_{\text{disc}}^F(h\cdot\One_{\cpt_f^F})$ can be viewed as a part of the spectral side  $J_{\text{spec}}^F(h\cdot\One_{\cpt_f^F})$ of Arthur's trace formula for $\SL(2)$ over $F$. In fact, we will show that the family of lattices $\{\Gamma_F\}_{F\in \F_{\signature}}$ has the spectral limit multiplicity property in the sense of Proposition \ref{prop:spectral:limit}, that is, $J_{\text{disc}}^F(h\cdot\One_{\cpt_f^F})$ constitutes the main term on the spectral side as $F$ varies over $\F_{\signature}$.
 
We then use the trace formula to equate the spectral side $J_{\text{spec}}^F(h\cdot\One_{\cpt_f^F})$ with the geometric side $J_{\text{geom}}^F(h\cdot\One_{\cpt_f^F})$. On the geometric side, the contribution from the center equals  
\[
\vol(G(F)\backslash G(\A_F)) \left(h(1)+ h(-1)\right)
\]
and constitutes the main term as $F$ varies over $\F_{\signature}$. Analogous to the spectral side, we are going to show that the family of lattices has the geometric limit multiplicity property, in the sense of Proposition \ref{prop:geom:limit}.
Putting all this together, we get that $\mu_{\Gamma_F}(\widehat{h_1})$ tends to 
\[
h(1)+h(-1)=h_1(1)=\mu_{{\rm Pl}}(\widehat{h_1}) 
\]
 as desired.

In this adelic formulation we prove the following statement which immediately implies Theorem~\ref{thm:effective} and also applies to $\GL(2)$:
\begin{proposition}
 Let $G=\SL(2)$ or $\GL(2)$ and  let $Z\subseteq G$ be the center of $G$. Let $G(\A_F)^1$ be the set of all $g\in G(\A_F)$ such that $|\det g|_{\A_F}=1$ (see \S\ref{sec:contr:elements} for the definition of $|\cdot|_{\A_F}$) and let $\cpt_\infty$ be the maximal compact subgroup of $G(\R^{\signature})^1$ defined in \S\ref{sec:contr:elements}. In generalization of \eqref{eq:def:volume} put
 \[
  \nu_F=\frac{\vol(G(F)\backslash G(\A_F)^1)}{\vol(\cpt_f^F)}.
 \]
 Then for every left and right $\cpt_\infty$--finite $h\in C_c^\infty(G(\R^{\signature})^1)$ and every $\eps>0$ we have 
 \[
\left|  \sum_{\pi}  \tr \pi_\infty (h) - \nu_F \sum_{z\in Z(F)\cap Z(\widehat{\OOO_F})} h(z)\right|
\ll_{h,\eps} \nu_F^{1-\delta_G +\eps},
 \]
where $\delta_{\SL(2)}=1/3$  and $\delta_{\GL(2)} = 1/4$. The $\pi$ run over all irreducible automorphic representations occurring in $L^2_\text{disc}(G(F)\backslash G(\A_F)^1)$ which are unramified at all non-archimedean places.
\end{proposition}

\begin{rem}
Theorem~\ref{thm:main} can also be viewed in the context of families of automorphic forms \cite{SaShTe14,ShTe16}. More precisely, let $A\subset \widehat{G(\R^{\signature})^1}_{\text{temp}}$ be a bounded subset with $\mu_{\text{Pl}}(\partial A)=0$. We define a family of automorphic forms as follows: For each $F\in \F_{\signature}$, let 
\[
\FFF(A,\Gamma_F)=\{\pi\subseteq L^2_{\text{disc}}(G(F)\backslash G(\A_F)^1)\mid \pi_{\infty}\in A\}. 
\]
Then $\FFF(A,\Gamma_F)$ is a finite set, and Theorem~\ref{thm:main} essentially provides an asymptotic count for the number of elements in $\FFF(A, \Gamma_F)$ in the sense that 
\[
 \frac{\left|\FFF(A,\Gamma_F)\right|}{\vol(G(F)\backslash G(\A_F)^1)}  \longrightarrow \mu_{\rm Pl}(A)
\]
as $F$ varies over $\F_{\signature}$. Here, as usual, we count the number of elements in $\FFF(A,\Gamma_F)$ with their multiplicities. 

In this picture we can incorporate the distribution of the Satake parameters at non-archimedean places. Namely, suppose that $S$ is a finite set of prime numbers. For each $p\in S$ fix a signature $\signature_p$ of $d$-dimensional extensions of $\Q_p$, and let $\signature_S = (\signature_p)_{p\in S}$.  Let $\F_{\signature, \signature_S}$ be the set of all $F\in \F_{\signature}$ such that $F_p=F\otimes\Q_p$ has signature $\signature_p$ at every $p\in S$. One can then study the sum 
\[
 \sum_{\pi} \tr\pi_{\infty\cup S} (h\cdot f_S)
\]
where $\pi$ runs over all discrete automorphic representations which are unramified outside of $\{v\mid v|\infty\text{ or }\exists p\in S:\ v|p\}$, $h\in \HHH(G(\R^{\signature})^1)$ and  $f_S\in C_c^\infty(G(F_S))$. Further, 
\[
\pi_{\infty\cup S}=\bigotimes_{v|\infty}\pi_v\otimes\bigotimes_{v| p, p\in S}\pi_v. 
\]
By our assumption, the groups $G(F_S)$ are isomorphic for all $F\in \F_{\signature, \signature_S}$ so that we can take the same test function $f_S$ for all $F$. One can then study the above sum as $F$ varies over $\F_{\signature, \signature_S}$.

However, there is a significant difference to the situation studied in \cite{Sh12,ShTe16,FiLaMu_limitmult,weyl2,KST16}: Since our family of lattices is not commensurable, we need to study an infinite sequence of trace formulas instead of using just one. More precisely, for each $F\in\F_{\signature}$ we need to study the trace formula for $G$ over $F$ and at the end compare these trace formulas with each other. This difference already becomes apparent in the situation studied in this paper compared to, e.g., \cite{FiLaMu_limitmult}.
\end{rem}

\section{Preliminaries}\label{sec:contr:elements}
In this section, $G$ equals $\SL(n)$ or $\GL(n)$ until further notice. For any number field $F$ we write $\A_F$ for the ring of integers of $F$, and $\A_F^1$ denotes the group of all $a\in\A_F^\times$ with $|a|_{\A_F}=1$, where $|\cdot|_{\A_F}$ denotes the adelic absolute value on $\A_F^\times$. If $v$ is any place of $F$, we write $|\cdot|_v$ for the norm on $F_v$, and let $|\cdot|_{\infty}=\prod_{v|\infty} |\cdot|_v$. We put $G(\A_F)^1=\{g\in G(\A_F)\mid |\det g|_{\A_F}=1\}$ and $G(\R^{\signature})^1=\{g\in G(\R^{\signature})\mid |\det g|_{\infty}=1\}$. 

\subsection{Maximal compact subgroups}
Let $F\in\F_{\signature}$ and $v$ be a place of $F$. If $v$ is non-archimedean, let $\OOO_{F_v}\subseteq F_v$ denote the ring of integers in $F_v$. Let
\[
\cpt_v^F=\cpt_{F_v}
=\begin{cases}
G(\OOO_{F_v})
&\text{if } v\text{ is non-archimedean},\\
\Orth(n)
&\text{if }G=\GL(n)\text{ and } v=\R,\\
\Unit(n)
&\text{if }G=\GL(n)\text{ and }v=\C,\\
\SO(n)
&\text{if }G=\SL(n)\text{ and } v=\R,\\
\SU(n)
&\text{if }G=\SL(n)\text{ and } v=\C
\end{cases}
\]
be the usual maximal compact subgroups in $G(F_v)$. 
 We further write $\cpt_f^F=\prod_{v<\infty}\cpt_v^F\subseteq G(\A_{F,f})$, $\cpt_{\infty}=\cpt_{\infty}^F=\prod_{v|\infty}\cpt_v^F\subseteq G(\R^{\signature})$, and $\cpt^F=\cpt_{\infty}^F\cdot\cpt_f^F$. Note that the group $\cpt_{\infty}$ does indeed not depend on $F$ but only on the signature $\signature$. We let $\HHH(G(\R^{\signature})^1)$ denote the space of all left and right $\cpt_\infty$-finite functions in $C_c^\infty(G(\R^{\signature}))$.

\subsection{Measures}\label{sec:measures}
We need to define measures on $G(\A_F)$ and its subgroups in a way that they are compatible when $F$ varies in $\F_{\signature}$. More precisely, using the fixed isomorphisms  $G(F_\infty)\simeq G(\R^{\signature})$ and $\cpt_\infty^F\simeq \cpt_{\infty}$, we get measures on $G(F_\infty)/\cpt_\infty^F\simeq G(\R^{\signature})/\cpt_{\infty}$ for every $F\in\F_{\signature}$. We therefore obtain canonical measures on every quotient 
$\Gamma_F\backslash G(F_\infty)^1/\cpt_\infty^F\simeq \Gamma_F\backslash G(\R^{\signature})^1/\cpt_{\infty}$, with $\Gamma_F=G(\OOO_F)$ for
$F\in \F_{\signature}$, where we used the counting measure on $G(\OOO_F)$, and used the isomorphism $G(\R^{\signature})\simeq \R_{>0}\times G(\R^{\signature})^1$ to define the measure on $G(\R^{\signature})^1$.
Now the choice of measures on $G(\A_F)^1$ and its subgroups gives a quotient measure on $G(F)\backslash G(\A_F)^1/\cpt^F$.
Since $G(F)\backslash G(\A_F)^1/\cpt^F$ equals a finite disjoint union $\bigsqcup \Gamma_F\backslash G(F_{\infty})^1/\cpt_{\infty}^F= \bigsqcup \Gamma_F\backslash G(\R^{\signature})^1/\cpt_{\infty}$, we need to make sure that they are compatible with the measures on $\Gamma_F\backslash G(F_\infty)^1/\cpt_\infty^F$ as constructed above.

Let $v$ be an archimedean or non-archimedean place of $\Q$, and let $E$ be a finite extension of $\Q_v$ (with $\Q_v:=\R$ if $v=\infty$). If $E$ is non-archimedean, let $\OOO_E$ denote the ring of integers in $E$. In that case we normalize the Haar measures on $E$ and $E^{\times}$ such that $\OOO_E$ and $\OOO_E^\times$ both have measure $\N(\partial_E)^{-\frac{1}{2}}$ where $\partial_E\subseteq\OOO_E$ denotes the different of extension $E/\Q_v$. If $E$ is archimedean, we take the usual Lebesgue measure $d_Ex$ on $E$, and $d_E x/|x|_E$ on $E^\times$. Note that $|x|_\C=x\bar{x}$ if $E=\C$.

The Haar measure on the maximal compact subgroups $\cpt_E$ is normalized such that $\cpt_E$ has measure $1$.

Now let $F$ be an arbitrary number field of degree $d=[F:\Q]$ and absolute discriminant $D_F$.
We take the product measures $dx=\prod_v dx_v$ and $d^{\times}x=\prod_v d^{\times}x_v$ on $\A_F$ and $\A_F^{\times}$, respectively.
Using the identification $\R_{>0}\ni t\mapsto (t^{1/d}, \ldots, t^{1/d},1,\ldots)\in F_{\infty}^{\times}\subseteq \A_F^{\times}$,  we get an isomorphism
$\A_F^{\times}\simeq \R_{>0}\times\A_F^1$
that also fixes a measure $d^{\times}b$ on $\A_F^1$ via
$d^{\times}x= d^{\times} b \frac{d t}{t}$
for $d^{\times} x$ the previously defined measure on $\A_F^{\times}$ and $\frac{dt}{t}=d^{\times} t$ the usual multiplicative Haar measure on $\R_{>0}$.
With this choice of measures we get
\[
\vol(F\backslash \A_F)=1
\;
\text{ and }
\;\;
\vol(F^{\times}\backslash \A_F^{1})=\res_{s=1}\zeta_F(s)
\]
where $\zeta_F(s)$ is the Dedekind zeta function of $F$
(cf. \cite[Chapter XIV, \S 7, Proposition 9]{La86}).

Let $Z\subseteq G$ be the center of $G$ and let $T_0\subseteq G$ denote the maximal split torus of diagonal matrices. Let $B\subseteq G$ be the Borel subgroup of all upper triangular matrices. Write $B=T_0U_0$. 
The above conventions also determine normalizations of measures on $T_0(\A_F)$, $T_0(\A_F)^1$, $T_0(F_v)$ (by identifying these groups with appropriate multiplicative groups via their coordinates), and  $U_0(\A_F)$ and  $U_0(F_v)$  (by identifying these groups with some affine space via the coordinates again).

The measure on $G(\A_F)$ is then normalized by using Iwasawa decomposition $G(\A_F)=U_0(\A_F) T_0(\A_F)\cpt^F =T_0(\A_F)U_0(\A_F)\cpt^F$: For any integrable function $f:G(\A_F)\longrightarrow\C$ we have
\begin{align*}
\int_{G(\A_F)}f(g)dg
&=\int_{\cpt^F}\int_{T_0(\A_F)}\int_{U_0(\A_F)}\delta_{0}(m)^{-1}f(umk)du\,dm\,dk\\
&=\int_{\cpt^F}\int_{T_0(\A_F)}\int_{U_0(\A_F)}f(muk)du\,dm \,dk
\end{align*}
(similarly for the groups $G(F_v)=U_0(F_v)T_0(F_v)\cpt_v^F$ over the local fields), where  $\delta_0=\delta_{P_0}$ is the modulus function for the adjoint action of $T_0$ on $U_0$.
If $G=\GL(n)$, we define a measure on $G(\A_F)^1$ via the exact sequence
\[
1\longrightarrow G(\A_F)^1\longrightarrow G(\A_F)\xrightarrow{\;g\mapsto|\det g|_{\A_F}\;}\R_{>0}\longrightarrow 1.
\]
To normalize measures on the adelic quotient we use the counting measure on the discrete groups $G(F)$.

We now specialize to $G=\SL(2)$ or $\GL(2)$. Then with the above measures we get
\begin{equation}\label{measures}
 \vol(G(F)\backslash G(\A_F)^1)=
\begin{cases}
 D_F^{1/2}\zeta_F(2)\res_{s=1}\zeta_F(s) 			&\text{if }G=\GL(2),\\
D_F^{1/2}\zeta_F(2)						&\text{if }G=\SL(2). 
 \end{cases}
\end{equation}

\begin{lemma}\label{lem:arch:compact}
We have  
\[
\vol(\cpt_f^F)=
\begin{cases}
 D_F^{-3/2} 						&\text{if }G=\GL(2),\\
D_F^{-1}						&\text{if }G=\SL(2). 
 \end{cases}
\]
where $\vol(\cpt_f^F)$ denotes the measure of $\cpt_f^F$ as a subset of $G(\A_{F,f})$ with respect to the measure on $G(\A_{F,f})$ defined above.
\end{lemma}
\begin{proof}
 Let $\chi:G(\A_{F,f})\longrightarrow\{0,1\}$ be the characteristic function of $\cpt_f^F$. Then
 \begin{multline*}
  \vol(\cpt_f^F) = \int_{G(\A_{F,f})} \chi(g)\, dg
  = \int_{\cpt_f^F}\int_{T_0(\A_{F,f})} \int_{U_0(\A_{F,f})} \chi(muk)\, du\, dm\, dk \\
  = \int_{T_0(\A_{F,f})} \int_{U_0(\A_{F,f})} \chi(muk)\, du\, dm.
 \end{multline*}
Now $\chi(mu)=1$ if and only if $m\in T(\A_{F,f})\cap \cpt_{f}^F = \prod_{v<\infty} T_0(\OOO_{F_v})$ and $u\in U_0(\A_{F,f})\cap\cpt_f^F= \prod_{v<\infty} U_0(\OOO_{F_v})$. Using the normalization of measures on $\OOO_{F_v}^\times$ and $\OOO_{F_v}$ we obtain the assertion.
\end{proof}

Setting $ \nu_F= \vol(G(F)\backslash G(\A_F)^1/\cpt_f^F)$  as in the introduction, we then get
\begin{equation}\label{eq:volume}
\nu_F= 
\begin{cases}
  D_F^{2}\zeta_F(2)\res_{s=1}\zeta_F(s) 			&\text{if }G=\GL(2),\\
D_F^{3/2}\zeta_F(2)						&\text{if }G=\SL(2). 
      \end{cases}
 \end{equation}

\subsection{Regulators and the class number formula}
We recall some facts about regulators of number fields and residues of Dedekind $\zeta$-functions. 
Let $F\in\F_{\signature}$, and let $\zeta_F(s)$ denote the Dedekind $\zeta$-function associated with $F$. Then $\zeta_F(s)$ has a simple pole at $s=1$ with residue
\[
\res_{s=1}\zeta_F(s)
=2^{r_1} (2\pi)^{r_2}\frac{h_F R_F}{w_F D_F^{1/2}}
\]
where
\begin{itemize}
\item $h_F$ is the class number of $F$,
\item $R_F$ is the regulator of $F$, and
\item $w_F$ is the number of roots of unity contained in $F$.
\end{itemize}
The rest of this section is devoted to stating bounds for the constant $\nu_F$ (defined in \eqref{eq:volume}) and the class number $h_F$, which we will need later.
For that, first note that $w_F$ can be bounded from above by the degree $d=[F:\Q]$.
By \cite{La70} the residue is bounded from above by
\begin{equation}\label{eq:upper:bound:residue}
\res_{s=1}\zeta_F(s)
\leq  (\log D_F)^{d-1}
\end{equation}
for every $F$ with $D_F\geq5$.

By \cite{Zi80} there exists an absolute constant $R_0>0$ such that for all number fields $F$ we have
\begin{equation}\label{eq:lower:bound:regulator}
R_F\geq R_0.
\end{equation}
Combining this with \eqref{eq:upper:bound:residue} we get an upper bound for the class number: For every $F\in\F_{\signature}$ we have
\begin{equation}\label{eq:upper:bound:class:number}
h_F\ll_{d} D_F^{1/2} (\log D_F)^{d-1}.
\end{equation}

Recall the definition of $\nu_F$ from \eqref{sec:measures}. Using above estimates, we obtain the following bounds for $\nu_F$: If $G=\GL(2)$, then
\[
 D_F^2\ll_d \nu_F \ll_d D_F^2 (\log D_F)^{d-1},
\]
and if $G=\SL(2)$, then
\[
 D_F^{3/2} \ll_d \nu_F \ll_d D_F^{3/2}.
\]

For later purposes we record the following consequences of above estimates:
\begin{lemma}\label{lem:quot:meas}
Let $G=\GL(2)$ or $\SL(2)$ and let $T_0\subseteq G$ be the diagonal torus as before.
We have 
 \[
  \frac{\vol(T_0(F)\backslash T_0(\A_F)^1)\vol(F\backslash \A_F)}{\vol(G(F)\backslash G(\A_F)^1)}
  \le \zeta(2d)^{-d} D_F^{-1/2} (\log D_F)^{d-1}
 \]
and
\[
   \frac{\vol(T_0(F)\backslash T_0(\A_F)^1)\vol(F\backslash \A_F)}{\vol(G(F)\backslash G(\A_F)^1)}
  \le \zeta(2d)^{-d} D_F^{-1} (\log D_F)^{d-1}
\]
for all $F\in\F_{\signature}$ with $D_F\ge5$. Here $\vol(H(F)\backslash H(\A_F))$ denotes the measure of the quotient  $H(F)\backslash H(\A_F)$ with respect to the measures on $H(\A_F)$  defined above for any of the groups $H$ considered above. (The measure on $H(F)$ is of course the counting measure.)
\end{lemma}
\begin{proof}
 For $\GL(2)$ as well as $\SL(2)$ we have 
 \[
    \frac{\vol(T_0(F)\backslash T_0(\A_F)^1\vol(F\backslash \A_F)}{\vol(G(F)\backslash G(\A_F)^1)}
    =\frac{\res_{s=1}\zeta_F(s)}{D_F^{1/2} \zeta_F(2)}
 \]
 and
 \[
    \frac{\vol(T_0(F)\backslash T_0(\A_F)^1\vol(\widehat{\OOO_F})}{\vol(G(F)\backslash G(\A_F)^1)}
    = \frac{\res_{s=1}\zeta_F(s)}{D_F \zeta_F(2)}.
 \]
Since $\zeta_F(2)\ge\zeta(2d)^d$ and $\res_{s=1}\zeta_F(s)\le (\log D_F)^{d-1}$ for all $F\in \F_{\signature}$ with $D_F\ge5$ the assertions follow. 
\end{proof}

\subsection{Suitably regular truncation parameter}\label{sec:suff:regular}
Let $\aaa$ denote the Lie algebra of $T_0(\R)\cap G(\R)^1$ and $A_0^G={\rm exp}\, (\aaa) \subseteq G(\R^{\signature})^1$. We identify $\aaa$ with the set $\{(s, -s)\mid s\in\R\}$. This fixes a Euclidean distance on $\aaa\simeq \R$ via $(s, -s)\mapsto s$.  
Let $\alpha\in\aaa^*=\Hom_\R(\aaa, \R)\simeq \{(r,-r)\mid r\in\R\}$  denote the (unique) positive root of $G$ with respect to $(T_0,B)$, and let $\varpi\in \aaa^*$ be the corresponding coroot. More explicitly, $\alpha=(1,-1)$ and $\varpi=\frac{1}{2}\alpha$.  
We put $\aaa^+=\{X\in\aaa\mid \alpha(X)>0\} = \{(s,-s)\mid s>0\}$.

In the construction of Arthur's trace formula, a \emph{suitably regular} truncation parameter $T\in\aaa^+$ plays a crucial role \cite{Ar78}. The property of being ``suitably regular'' means that $T$ is sufficiently far away from the walls of $\aaa^+$ in a way that depends on the ground field and on  the support of the test function used in the trace formula. Our test function will be of the form $f=f_F=f_\infty\cdot \One_{\cpt_f^F}\in C_c^\infty(G(\A_F)^1)$, where $f_\infty\in C_c^\infty(G(\R^{\signature}))$ is fixed independently of the field $F\in \F_{\signature}$ and $\One_{\cpt_f^F}\in C_c^\infty(G(\A_{F,f}))$ denotes the characteristic function of $\cpt_f^F$. In this case the regularity condition can be made explicit in $F$ \cite[\S7]{coeff_est}: There exists $\rho>0$ depending only on the support of $f_\infty$ and the degree $d$ of $F$ over $\Q$ such that $T\in\aaa^+$ is suitably regular in Arthur's sense if
\begin{equation}\label{eq:suitably:regular}
\alpha(T)\geq \rho\max\left\{1,\log D_F\right\}.
\end{equation}
Since we always assume that the degree $d$ is at least $2$, we can replace this inequality by $\alpha(T)\geq \rho \log D_F$.
(Note that for $\GL(n)$ and $\SL(n)$, $n>2$, a similar assertion holds but $\alpha(T)$ on the left hand side of the inequality has to be replaced by the minimum of $\beta(T)$ with $\beta$ running over the positive roots.)

\section{Geometric limit multiplicity property}\label{sec:geom:limit}

Let $F\in\F_{\signature}$ and let $G=\SL(2)$ or $\GL(2)$ as an algebraic group over $F$. We denote by $\Jgeom^F$  the geometric side of Arthur's trace formula for $G$ over $F$ as in \cite{Ar05}. We first show the geometric limit property in the following form: 

\begin{proposition}\label{prop:geom:limit}

Let $\One_{\cpt_f^F}\in C_c^{\infty} (G(\A_{F,f}))$ the characteristic function of $\cpt_f^F$.
Then for every $f_{\infty}\in C_c^{\infty}(G(\R^{\signature})^1)$ we have
\[
\lim_{F\in\F_{\signature}}\left(\frac{\Jgeom^F(f_{\infty}\cdot \One_{\cpt_f^F})}{\vol(G(F)\backslash G(\A_F)^1)}- \sum_{z\in Z(F)}(f_{\infty}\cdot\One_{\cpt_f^F})(z)\right)=0,
\]
where $Z\subseteq G$ denotes the center of $G$. 
\end{proposition}

In fact, we shall prove the following effective estimate:
\begin{proposition}\label{prop:upper:bound}
For every $f_{\infty}\in C_c^{\infty}(G(\R^{\signature})^1)$ and every $\eps>0$ there exists a constant $c>0$ such that for all $F\in\F_{\signature}$ we have
\begin{equation}\label{eq:upper:bound}
\left| \frac{\Jgeom^F(f_{\infty}\cdot \One_{\cpt_f^F} )}{ \vol(G(F)\backslash G(\A_F)^1)}-\sum_{z\in Z(F)}(f_{\infty}\cdot \One_{\cpt_f^F})(z) \right|
\le c D_F^{-\frac{1}{2}+\eps}.
\end{equation}
\end{proposition}

The proof of these two propositions will occupy the next few sections. 

If $f_{\infty}\in C_c^\infty(G(\R^{\signature})^1)$ we write $f_F=f_{\infty}\cdot \One_{\cpt_f^F}$. Recall from \cite{Ar05} that $\Jgeom^F(f_F)$ is the value at $T=0$ of a polynomial $\Jgeom^{F,T}(f_F)$ of degree $1$ in $T\in\aaa$. 
Let $J_{\text{geom}\smallsetminus Z}^{F,T}(f_F)$ denote $J_{\text{geom}}^{F,T}(f_F)$ with the central contribution removed, that is,
\[
 J_{\text{geom}\smallsetminus Z}^{F,T}(f_F)
 =J_{\text{geom}}^{F,T}(f_F)- \sum_{z\in Z(F)}\vol(G(F)\backslash G(\A_F)^1) f_F(z).
\]
Let $F(\cdot,T): G(\A_F)^1\longrightarrow\C$ be the truncation function as defined in \cite[\S 6]{Ar78}.  More precisely, in our situation $F(\cdot, T)$ equals the characteristic function of all $g\in G(\A_F)^1$ with  
\[
 \varpi(H_0(\delta g)-T) \leq 0
\]
for every $\delta\in G(F)$. Here $H_0:G(\A_F)^1\longrightarrow\aaa$ is the Iwasawa projection, namely if $g\in G(\A_F)^1$ with Iwasawa decomposition $g=tuk\in T_0(\A_F)U_0(\A_F)\cpt^F$, $t=\diag(t_1, t_2)$, then $H_0(g)= (\log|t_1|_{\A_F}, \log|t_2|_{\A_F})$.

We introduce auxiliary distributions
\begin{equation}\label{eq:def:truncated:integral}
 j^{F,T}(f_F)=\int_{G(F)\backslash G(\A_F)^1} F(x,T)\sum_{\gamma\in G(F)} f_F(x^{-1}\gamma x)\,dx,
\end{equation}
and
\begin{equation}\label{eq:def:truncated:int:minus:center}
 j^{F,T}_{G\smallsetminus Z}(f_F)=\int_{G(F)\backslash G(\A_F)^1} F(x,T)\sum_{\gamma\in G(F)\smallsetminus Z(F)} f_F(x^{-1}\gamma x)\,dx
\end{equation}
for $T$ suitably regular. 
The distribution $\Jgeom^{F,T}(f_F)$ (resp. $J_{\text{geom}\smallsetminus Z}^{F,T}(f_F)$) can be approximated by $j^{F,T}(f_F)$ (resp. $j_{G\smallsetminus Z}^{F,T}(f_F)$) exponentially good in $T$ \cite[Theorem~2]{Ar79}.
We need to make this approximation explicit with respect to the field $F$, see Lemma~\ref{lem:approx} below.

The main step in proving the geometric limit multiplicity is the following:
\begin{lemma}\label{lem:upper:bound:T}
Let $\rho$ be as in \eqref{eq:suitably:regular}. 
For every $f_{\infty}\in C_c^{\infty}(G(\R^{\signature})^1)$ there exists $c>0$ such that for all $F\in \F_{\signature}$ we have
\begin{equation}\label{eq:upper:bound:T}
\frac{\left|j^{F,T}_{G\smallsetminus Z}(f_F)\right|}{\vol(G(F)\backslash G(\A_F)^1)}
\le c D_F^{-\frac{1}{2}} (\log D_F)^{2d} \varpi(T)
\end{equation}
for all $T\in\aaa$ with $\alpha(T)\geq \rho\log D_F$.
\end{lemma}
Proposition \ref{prop:upper:bound} will then follow from this lemma by an interpolation argument for polynomials. We shall prove this lemma and the previous two propositions in Section~\ref{sec:proof}. Sections~\ref{sec:contr:elements}--\ref{sec:unip:contr} contain auxiliary results for the proof of Lemma \ref{lem:upper:bound:T}. The arguments are in general the same for $\GL(2)$ and $\SL(2)$ so that we treat both cases at once unless stated otherwise.

\begin{rem}
 For $\GL(2)$ a more direct approach via the explicit expansion of the geometric side of the trace formula could be used (see, e.g., \cite{GeJa79}). However, for $\SL(2)$ it seems easier to use the distribution   $j^{F,T}_{G\smallsetminus Z}(f_F)$. The latter approach seems also more suitable for generalizations to higher rank, though this has not been successful so far.
\end{rem}

\section{Auxiliary results on contributing elements}
In this section we allow $G=\GL(n)$ or $\SL(n)$ with $n\ge2$ arbitrary until further notice. 
\subsection{Test functions}\label{sec:test:fcts}

We fix $f_{\infty}\in C_c^{\infty}(G(\R^{\signature})^1)$. For the proof of the geometric limit multiplicity property we may assume that $f_{\infty}\geq0$ whenever convenient. We further assume without loss of generality that $f_{\infty}$ is conjugation invariant by every element in $\cpt_{\infty}$. Let $R\geq1$ be such that the support of $f_{\infty}$ is contained in
\[
B_R^{\signature}= \{g=(g_{ij})_{i,j=1,\ldots,n}\in \Mat_{n\times n}(\R^{\signature})\mid \forall ~i,j\in\{1,\ldots,n\}:~\|g_{ij}-\delta_{ij}\|_{\signature}\leq R\},
\]
where $\delta_{ij}$ denotes the Kronecker delta in the ring $\R^{\signature}$, i.e., $\delta_{ii}=(1,\ldots,1)\in\R^{r_1}\oplus\C^{r_2}$ and  $\delta_{ij}=(0,\ldots, 0)$ for $i\neq j$. Further, for $v=(v_1, \ldots, v_{r_1},v_{r_1+1}, \ldots, v_{r_1+r_2})\in\R^{\signature}$,
\[
\|v\|_{\signature}=\left(\sum_{i=1}^{r_1} v_i^2+ 2\sum_{i=1}^{r_2} v_{r_1+i} \overline{v_{r_1+i}}\right)^{\frac{1}{2}}
 \]
with $\overline{v_{r_1+i}}$ the complex conjugate of $v_{r_1+i}$.

As before, we set $f_F=f_{\infty}\cdot\One_{\cpt_f^F}\in C_c^{\infty}(G(\A_F)^1)$, $F\in\F_{\signature}$.

\subsection{Geometric equivalence classes}

Recall the definition of the geometric equivalence classes $\ooo\subseteq G(F)$ from \cite[\S 10]{Ar05}: Two elements $\gamma_1, \gamma_2\in G(F)$ lie in the same class $\ooo$ if and only if the two semisimple parts $\gamma_{1,s}$, $\gamma_{2,s}$ of $\gamma_1$ and $\gamma_2$ in their Jordan decomposition are conjugate in $G(F)$.

In our case we can formulate this in terms of characteristic polynomials of semisimple conjugacy classes: There is a canonical map from the set of equivalence classes $\OOO^{G(F)}$ in $G(F)$ to the set of
\begin{itemize}
\item monic polynomials of degree $n$ with coefficients in $F$ and non-zero constant term if $G=\GL(n)$, 
\item monic polynomials of degree $n$ with coefficients in $F$ and constant term equal to $(-1)^n$ if $G=\SL(n)$,
\end{itemize}
by mapping the semisimple conjugacy class attached to $\ooo$ to its characteristic polynomial $\chi_\ooo$.

This map is a bijection if $G=\GL(n)$. 
If $G=\SL(2)$, the situation is more complicated. If the polynomial is split over $F$, that is, it is of the form $(X-\xi)(X-\xi^{-1})$ for some $\xi\in F^{\times}$, then there are two (resp.\ one) corresponding classes $\ooo\in\OOO^{\SL_2(F)}$ if $\xi\neq\pm1$ (resp.\ $\xi=\pm1$). If the polynomial is irreducible over $F$, there might be more classes but their number can be bounded, see Lemma~\ref{lem:classes} below.

\begin{proposition}\label{prop:finitely:many:classes}
Let $f_\infty\in C_c^\infty(G(\R^{\signature}))$. 
There exist finitely many polynomials $\chi_1(X), \ldots, \chi_s(X)\in\bar{\Q}[X]$ (with $\bar\Q$ some fixed algebraic closure of $\Q$)
such that for all $F\in\F_{\signature}$ we have
\begin{equation}\label{eq:fin:conj:classes}
j_{G\smallsetminus Z}^{F,T}(f_F)
=\int_{G(F)\backslash G(\A_F)^1} F(x,T)\sum_{\gamma\in \Sigma(F)\smallsetminus Z(F)} f_F(x^{-1}\gamma x)\,dx
\end{equation}
for all $T\in\aaa^+$ with $\alpha(T)\geq \rho\log D_F$,
where
\begin{equation}\label{eq:def:sigma}
 \Sigma(F)
 = \bigsqcup_{\substack{\ooo\in \OOO^{G(F)}: \\ \chi_{\ooo}\in\{\chi_1, \ldots, \chi_s\}}}\ooo
\end{equation}
is the union over all
geometric equivalence classes in $G(F)$ whose characteristic polynomial equals one of $\chi_1, \ldots, \chi_s$. The set $\{\chi_1,\ldots,\chi_s\}$ depends on the support of the function $f_{\infty}$.
\end{proposition}

\begin{rem}
By our choice of test function, namely the non-archimedean part $\One_{\cpt_f^F}$, only elements in $\ooo\in\OOO^{G(F)}$ for which the coefficients of $\chi_{\ooo}(X)$ are algebraic integers can contribute non-trivially to the sum-integral in \eqref{eq:fin:conj:classes}. Hence we are going to assume that every polynomial $\chi_j(X)$, $j=1,\ldots,s$, has coefficients which are algebraic integers. 
\end{rem}

Before we prove Proposition~\ref{prop:finitely:many:classes} we need the following lemma:
\begin{lemma}\label{lem:succ:min:bound:below}
There exists a constant $c>0$ depending only on $d$ such that for every number field $K\neq \Q$ of degree $d_K\leq d$ and every $x\in \OOO_K$ with $K=\Q(x)$ we have
\[
 \|x\|_{\signature^K} \geq c \max_{\Q\subsetneq E\subseteq K} D_E^{4/d^3},
\]
where the maximum is taken over all primitive number fields $\neq\Q$ contained in $K$ and $\signature^K$ denotes the archimedean signature of $K$. (We recall that a number field is primitive if it does not contain any non-trivial subfield.)
\end{lemma}
\begin{proof}
Let $K$ be as in the lemma and $x\in \OOO_K$ an element generating $K$ over $\Q$. Let $E\subseteq K$, $E\neq\Q$, be any primitive subfield, and let $m=[K:E]=d_K/d_E$. Then $K=E(x)$ and the characteristic polynomial $\chi(X)$ of $x$ over $E$ has degree $m$. Let $x_1, \ldots, x_m\in \OOO_{\bar K}$ be the roots of $\chi$ in some algebraic closure $\bar K$ of $K$. Since $x$ generates $K$ over $\Q$ and $E\neq\Q$, one of the coefficients of $\chi$ has to be an element in $\OOO_E\minus\Z$. Let $s$ denote the elementary symmetric polynomial in the roots of $\chi$ which corresponds to this coefficient so that
\[
s(x_1, \ldots, x_m)\in \OOO_E\minus\Z.
\]
Let $1\leq k\leq m$ be the degree of $s$. Then by an elementary calculation one gets
\[
\|s(x_1, \ldots, x_m)\|_{\signature^E}^2
\leq c_1 \|x\|_{\signature^K}^k
\]
for $c_1>0$ some constant depending only on $d_K$, and also
\[
\|\frac{1}{d_E}\tr_{E/\Q} s(x_1, \ldots, x_m)\|_{\signature^E}^2
\leq c_2 \|s(x_1, \ldots, x_m)\|_{\signature^E}^2
\leq c_1c_2 \|x\|_{\signature^K}^k
\]
for $c_2>0$ another constant depending only on $d_K$.

Let $\OOO_E^0=\{a\in\OOO_E\mid \tr_{E/\Q}a=0\}$, and let $\signature^E$ be the signature of $E$. Let $Q_E$ denote the quadratic form $Q_E(x)=\|x-\frac{1}{d_E}\tr_{E/\Q} x\|_{\signature^E}^2$.
Then for some $c_3>0$ depending only on $d$ we have
\[
\|x\|_{\signature^K}^2
\geq c_3 Q_E(s(x_1, \ldots, x_m))^{4/k}
\]
and $s(x_1, \ldots, x_m)\not\in\Z$.

 By \cite[Lemma 3.1]{ElVe06} there exists $c_4>0$ depending only on $d_E$ such 
 \[
 Q_E(a)\geq c_4 D_E^{2/(d_E(d_E-1))} \geq c_5 D_E^{2/d_E^2}
 \]
 for every $a\in\OOO_E^0\smallsetminus\{0\}$. Now if $a\in\OOO_E\smallsetminus\Z$ is arbitrary, $d_E (a-\frac{1}{d_E}\tr_{E/\Q}a)\in\OOO_E^0\minus\{0\}$ so that
for $x$ as above
\begin{align*}
\|x\|_{\signature^K}
\geq c_3^{-1} Q_E(s(x_1, \ldots, x_m))^{2/k}
& = c_3^{-1} d_E^{-4/k} Q_E(d_Es(x_1, \ldots, x_m))^{2/k} \\
& \geq c_6 D_E^{\frac{4}{d_E^2 k}}
\geq c_6 D_E^{\frac{4}{d^3}}
\end{align*}
 for $c_6>0$ some constant depending only on $d$ as asserted.
\end{proof}

\begin{proof}[Proof of Proposition \ref{prop:finitely:many:classes}]
Let $\Xi_F:G(\A_F)\longrightarrow \A_F^n$ be the map associating with  $\gamma\in G(\A_F)^1$, the tuple $(a_0,a_1, \ldots, a_{n-1})$ of the coefficients of the characteristic polynomial $X^n+a_{n-1}X^{n-1}+\ldots+a_1X+a_0$ of $\gamma$. 
 The image of
\[
\left(\bigcup_{\gamma\in G(F)} \{g^{-1}\gamma g\mid g\in G(\A_F)^1\}\right)
\cap \supp f_F
\]
under $\Xi_F$ must be contained in
\[
\left(F\cap B_{NR^n}^{\signature}\cdot\widehat{\OOO_F}\right)^n
= \left(\OOO_F\cap B_{NR^n}^{\signature}\right)^n
\]
for some constant $N$ depending on $n$.
We claim that there exists a discrete subset $\Lambda\subseteq \R^{\signature}$ such that for all fields $F\in\F_{\signature}$ we have
\begin{equation}\label{eq:intersection}
\left(\OOO_F\cap B_{NR^n}^{\signature}\right)^n
=\left(\Lambda  \cap B_{NR^n}^{\signature}\right)^n.
\end{equation}
To prove this claim define for $D_0>0$ ($D_0$ will be determined later) the following objects:
\begin{itemize}
\item
$\EEE_{D_0}^0$ denote the set of all primitive fields occurring as subfields of elements in $\F_{\signature}$ with absolute discriminant $\leq D_0$ (including $\Q$). Here a primitive field is a number field with no non-trivial subfields.

\item
$\EEE_{D_0}$ denote the set of all subfields of elements of $\F_{\signature}$ which are composites of elements of $\EEE_{D_0}^0$.
\end{itemize}

 $\EEE_{D_0}^0$ and $\EEE_{D_0}$ are both finite sets and $\EEE_{D_0}^0\subseteq\EEE_{D_0}$. For every $E\in\EEE_{D_0}$ the ring of integers $\OOO_E$ is a discrete subset of $\R^{\signature}$, and we have an embedding $E_{\infty}=E\otimes\R\hookrightarrow\R^{\signature}$. We claim that we can choose $D_0$ such that \eqref{eq:intersection} holds with
\[
 \Lambda:=\bigcup_{E\in\EEE_{D_0}} \OOO_E.
\]
For this we use Lemma \ref{lem:succ:min:bound:below} and the notation therein: Let $x\in\OOO_F\minus(\Lambda\cap \OOO_F)$, and let $K=\Q(x)\subseteq F$ be the subfield of $F$ generated by $x$. Since $x\not\in\Lambda$ there exist a primitive subfield $E\subseteq K$ of absolute discriminant $D_E>D_0$. By Lemma~\ref{lem:succ:min:bound:below} there exists $c>0$ depending only on $d$ such that
\[
\|x\|_{\signature^F}=[F:K]\|x\|_{\signature^K}
 \geq c D_E^{4/d^3}.
\]
Now suppose that also $x\in B^{\signature}_{NR^n}$. 
Then 
\[
NR^n\geq\|x\|_{\signature}
\geq c D_E^{4/d^3} 
> cD_0^{4/d^3} .  
\]
Choosing any $D_0>\left(\frac{NR^n}{c}\right)^{d^3/4}$ 
leads to a contradiction, hence proving our claim that
\[
\left(\OOO_F\minus(\Lambda\cap \OOO_F)\right)\cap B_{NR^n}^{\signature}
\]
is empty for every $F\in\F_{\signature}$.

Since the set $\Lambda^n$ is discrete in $(\R^{\signature})^n$, the set
\[
\left(\Lambda \cap B_{NR^n}^{\signature}\right)^n
\]
 is finite. Writing $\chi_1, \ldots, \chi_s$ for the degree $n$ polynomials corresponding to these points, the proposition follows.
\end{proof}

Now let $G=\SL(2)$ or $\GL(2)$. 
We fix $\Sigma_0:=\{\chi_1, \ldots, \chi_s\}$ as in Proposition \ref{prop:finitely:many:classes} from now on and let $\Sigma(F)$ be defined as in \eqref{eq:def:sigma}. For each $F\in \F_{\signature}$ we define three disjoint subsets of $\Sigma_0$, and divide $\Sigma(F)$ accordingly:
\begin{itemize}
 \item
$\Sigma_0(F)_{\text{reg.ell.}}$ is the set of those elements in $\Sigma_0$ which have coefficients in $F$, and which are irreducible over $F$

\item
$\Sigma_0(F)_{\text{reg.split}}$ is the set of those elements in $\Sigma_0$ which have coefficients in $F$, and which split over $F$ into two distinct linear factors.

\item $\Sigma_0(F)_{\text{unip}}$ is the set of those elements in $\Sigma_0$ which have coefficients in $F$, and which are the square of a linear factor.
\end{itemize}
In each of the three cases, we define  $\Sigma(F)_*\subseteq\Sigma(F)$ to be is the union over all equivalence classes in $\OOO^{G(F)}$ whose characteristic polynomial is contained in $\Sigma_{0}(F)_*$. In particular, all central elements of $\Sigma(F)$ are contained in $\Sigma(F)_{\text{unip}}$ so that we can write
\[
\Sigma(F)\smallsetminus Z(F)
=\Sigma(F)_{\text{reg.ell.}} \sqcup \Sigma(F)_{\text{reg.split}}\sqcup \left(\Sigma(F)_{\text{unip}}\smallsetminus Z(F)\right).
\]
We decompose the sum-integral $j_{G\smallsetminus Z}^{F,T}(f_F)$ in \eqref{eq:fin:conj:classes} according to this decomposition of $\Sigma(F)\smallsetminus Z(F)$ as
\begin{equation}\label{eq:splitting}
 j_{G\smallsetminus Z}^{F,T}(f_F)
=: j_{\text{reg.ell.}}^{F,T}(f_F) + j_{\text{reg.split}}^{F,T}(f_F) + j_{\text{unip}\smallsetminus Z}^{F,T}(f_F),
\end{equation}
and treat each of the summands separately in the following sections.

\begin{lemma}\label{lem:classes}
Let $G=\SL_2$. There exists a constant $c>0$ depending only on  the degree $d$ of $\signature$ and $f_\infty$ such that the following holds: If $F\in \F_{\signature}$ and $\chi\in\Sigma_0(F)_{\text{reg.ell}}$, then the number of $\ooo\in \OOO^{\SL_2(F)}$  which have characteristic polynomial $\chi$ is bounded by $c$.
\end{lemma}
\begin{proof}
Let $F\in \F_{\signature}$ and $\chi\in\Sigma_0(F)_{\text{reg.ell}}$.  Let $\Xi_\chi(F)$ denote the set of all $\gamma\in G(F)$ with characteristic polynomial $\chi$.  Since $\chi$ is irreducible over $F$, the equivalence classes $\ooo\in \OOO^{G(F)}$ with characteristic polynomial $\chi$ are in fact $G(F)$-conjugacy classes in $\Xi_\chi(F)$. Note that $\Xi_\chi(F)$ is the stable conjugacy class of any $\gamma\in\Xi_{\chi}(F)$. Hence we need to show that the number of $G(F)$-conjugacy classes in $\Xi_\chi(F)$ is finite and bounded independently of $F$. 
 
 Instead of $G(F)$-conjugacy, we consider $G(\A_F)$-conjugacy in $\Xi_\chi(F)$. This suffices for our purposes, since by \cite[Lemma~8.6]{ShTe16} the number of $G(F)$-conjugacy classes mapping to the same $G(\A_F)$-conjugacy class is bounded by an absolute constant (independent of $F$).
 
Let $\gamma$ be the companion matrix of $\chi$. Then $\gamma\in \Xi_\chi(F)$. 
Let $\mathfrak{g}$ and $\mathfrak{g}_\gamma$ denote the Lie algebras of $G$ and of the centralizer  $G_\gamma$ of $\gamma$, respectively. Let 
\[
 D^G(\gamma)= \det\left(1-\Ad(\gamma)\right)_{\mathfrak{g}(F)/\mathfrak{g}_\gamma(F)}\in F^{\times}
\]
be the Weyl discriminant of $\gamma$,
and let $S_1$ denote the set of all non-archimedean places of $F$ such that $|D^G(\gamma)|_v\neq1$. Let $S_2$ denote the smallest set of all places of $F$ (including all archimedean ones) such that for any $v\not\in S_2$ we have $\gamma_v\in\cpt_v^F$. Let $S=S_1\cup S_2$. The set $S$ is finite, and the number of elements in $S$ can be bounded in terms of $\chi$ only (i.e., independently of $F$). What is more, there are is a finite set of prime numbers $P$ depending only on $\chi$ but not on $F$ such that if $v\in S$ is a non-archimedean place, then $v|p$ for some $p\in P$.

To bound the number of $G(\A_F)$-conjugacy classes in $\Xi_\chi(F)$, note first that if $v\not\in S$ and $\delta\in \Xi_\chi(F)$, then $\gamma_v$ and $\delta_v$ are conjugate in $G(F_v)$ by \cite[\S8]{Ko86}. 

If $v\in S$, then by the remark at the end of the last paragraph $F_v$ equals $\R$, $\C$, or an extension of $\Q_p$ of degree $\le d$ for some $p\in P$, that is, $F_v$ is an element in a finite collection of local fields which depends only on $\chi$. The number of $G(F_v)$-conjugacy classes in $\Xi_\chi(F)$ is finite \cite{Ko86}. It follows that there exists $c_1>0$ independent of $F$ such that the number of $G(\A_F)$-conjugacy classes in $\Xi_\chi(F)$ is bounded by $c_1$. Together with the previous remark on the relation between $G(F)$- and $G(\A_F)$-conjugacy classes, the assertion of the lemma follows.

\end{proof}

\section{The regular elliptic contribution}
From now on, we restrict to the two groups $G=\SL(2)$ and $G=\GL(2)$. 
 We start with bounding the regular elliptic contribution in \eqref{eq:splitting}.
Without loss of generality we assume from now on that the test function is non-negative, that is, $f_F\ge0$.
We have
\[
 j_{\text{reg.ell.}}^{F,T}(f_F)
\leq \int_{G(F)\backslash G(\A_F)^1} \sum_{\gamma\in \Sigma(F)_{\text{reg.ell.}}} f_F(x^{-1}\gamma x)\,dx,
\]
since $0\leq F(x,T)\leq 1$, and the right hand side converges.  Hence
\[
 j_{\text{reg.ell.}}^{F,T}(f_F)
\leq \sum_{[\gamma]\subseteq\Sigma(F)_{\text{reg.ell.}}}\vol(G_{\gamma}(F)\backslash G_{\gamma}(\A_F)^1) \int_{G_{\gamma}(\A_F)\backslash G(\A_F)} f_F(x^{-1}\gamma x)\,dx,
\]
where the sum runs over $G(F)$-conjugacy classes $[\gamma]$ in $\Sigma(F)_{\text{reg.ell.}}$.
Each of the orbital integrals can be factorized as
\begin{multline*}
 \int_{G_{\gamma}(\A_F)\backslash G(\A_F)} f_F(x^{-1}\gamma x)\,dx\\
=\int_{G_{\gamma}(\R^{\signature})\backslash G(\R^{\signature})} f_{\infty}(x^{-1}\gamma x)\,dx
\cdot\prod_{v<\infty}\int_{G_{\gamma}(F_v)\backslash G(F_v)} \One_{\cpt_v^F}(x^{-1}\gamma x)\,dx
\end{multline*}
with the product running over all non-archimedean places $v$ of $F$.
Note that the archimedean orbital integral can take values in a fixed finite set which is independent of $F\in\F_{\signature}$.

We keep our assumption that the characteristic polynomial of $\gamma$ has coefficients in $\OOO_F$.

\begin{lemma}\label{lem:orb:int:elliptic}
 There is a constant $\eta>0$ independent of $F$ such that for every $\gamma\in\Sigma(F)_{\text{reg.ell}}$ and every non-archimedean place $v$ of $F$ the following holds: Let $E_v/F_v$ denote the $F_v$-algebra $F_v\times F_v$ if $\gamma$ splits over $F_v$, and let $E_v$ be the quadratic splitting field of $\gamma$ over $F_v$ if $\gamma$ is non-split over $F_v$. 
 Then:
 \begin{enumerate}[label=(\roman{*})]
 \item If $\mathbf{G=GL(2)}$: For every non-archimedean place $v$ of $F$ we have
\[
\int_{G_{\gamma}(F_v)\backslash G(F_v)} \One_{\cpt_v}(x^{-1}\gamma x)\,dx
\leq \Delta_v(\gamma)^{-1} \frac{\vol(\OOO_{F_v})\vol(\OOO_{F_v}^{\times})^2}{\vol(\OOO_{E_v}^{\times})}
\]
where
\[
 \Delta_v(\gamma)= |(\tr\gamma)^2-4\det\gamma|_v
\]
is the $v$-adic absolute value of the discriminant of $\gamma$. In particular,
\[
 \int_{G_{\gamma}(\A_{F,f})\backslash G(\A_{F,f})} \One_{\cpt_f^F}(x^{-1}\gamma x)\,dx
\ll_{\signature, f_\infty} \Delta_{\signature}(\gamma) \frac{\vol(\widehat{\OOO_F})\vol(\widehat{\OOO_F}^{\times})^2}{\vol(\widehat{\OOO_{F(\gamma)}}^{\times})}
\]
for $\Delta_{\signature}(\gamma)=|(\tr\gamma)^2-4\det\gamma|_{\signature}$, $F(\gamma)$ the quadratic splitting field of $\gamma$ over $F$.

\item If $\mathbf{G=SL(2)}$: Let $\kappa_v=1$ unless $E_v/F_v$ is ramified quadratic field extension in which case $\kappa_v=2$. Define a norm $\NNN:E_v^{\times}\longrightarrow F_v^{\times}$ by $x\mapsto x_1x_2$ if $x=(x_1, x_2)\in E_v\simeq F_v\times F_v$, and by $x\mapsto x\bar{x}$ if $E_v/F_v$ is a quadratic field extension and $\bar{\cdot}:E_v\longrightarrow E_v$ denotes the non-trivial $F_v$-linear involution of $E_v$. Then for every non-archimedean place $v$ of $F$ we have
\[
\int_{G_{\gamma}(F_v)\backslash G(F_v)} \One_{\cpt_v}(x^{-1}\gamma x)\,dx
\leq \kappa_v \Delta_v(\gamma)^{-1} \frac{\vol(\OOO_{F_v})\vol(\OOO_{F_v}^{\times})}{\vol(\OOO_{E_v}^{(1)})},
\]
where $\OOO_{E_v}^{(1)}=\{x\in \OOO_{E_v}\mid \NNN(x)=1\}$, and
\[
\int_{G_{\gamma}(\A_{F,f})\backslash G(\A_{F,f})} \One_{\cpt_f^F}(x^{-1}\gamma x)\,dx
\ll_{\signature, f_\infty}  \Delta_{\signature}(\gamma) \frac{\vol(\widehat{\OOO_F})\vol(\widehat{\OOO_F}^{\times})}{\vol(\widehat{\OOO_{F(\gamma)}}^{(1)})},
\]
where $\widehat{\OOO_{F(\gamma)}}^{(1)}=\prod_{v<\infty} \OOO_{E_v}^{(1)}$.
\end{enumerate}
\end{lemma}
\begin{proof}
We treat the cases $G=\GL(2)$ and $G=\SL(2)$ simultaneously unless noted otherwise.

If $\gamma$ splits over $F_v$, that is, if $\gamma$ is $G(F_v)$-conjugate to a diagonal matrix $\diag(\gamma_1, \gamma_2)$, $\gamma_1\neq\gamma_2$, then
\[
  \int_{G_{\gamma}(F_v)\backslash G(F_v)} \One_{\cpt_v^F}(x^{-1}\gamma x)\,dx
=\int_{U_0(F_v)} \One_{\cpt_v^F}(u^{-1}\gamma u)\,du
=|\gamma_1-\gamma_2|_v^{-1}\vol(\OOO_{F_v}).
\]
Note that in the split case $\OOO_{E_v}^{\times}=\OOO_{F_v}^{\times}\times\OOO_{F_v}^{\times}$, and $\OOO_{E_v}^{(1)}\simeq \OOO_{F_v}^{\times}$. Further, $|\gamma_1-\gamma_2|_v^{-1}=\Delta_v(\gamma)^{-1/2}\leq \Delta_v(\gamma)^{-1}$ since the characteristic polynomial of $\gamma$ can be assumed to have integral coefficients (otherwise the integral vanishes as remarked before). This proves the lemma in the split case for $G=\GL(2)$ as well as $G=\SL(2)$.

If $\gamma$ is non-split over $F_v$, $Z(F_v)\backslash G_{\gamma}(F_v)$ is compact, and we can write
\[
 \int_{G_{\gamma}(F_v)\backslash G(F_v)} \One_{\cpt_v^F}(x^{-1}\gamma x)\,dx
 =\vol(Z(F_v)\backslash G_{\gamma}(F_v))^{-1}  \int_{Z(F_v)\backslash G(F_v)} \One_{\cpt_v^F}(x^{-1}\gamma x)\,dx.
\]
Now $G_{\gamma}(F_v)$ can be identified with the quadratic splitting field $E_v$ of $\gamma$ over $F_v$ so that we can compute
\begin{multline*}
\vol(Z(F_v)\backslash G_{\gamma}(F_v))\\
=
\begin{cases}
 \vol(F_v^{\times}\backslash E_v^{\times}) 
=\left.\begin{cases}
\vol(\OOO_{F_v}^{\times}\backslash \OOO_{E_v}^{\times})						&\text{if }E_v/F_v\text{ is unramified,}\\
2\vol(\OOO_{F_v}^{\times}\backslash \OOO_{E_v}^{\times})						&\text{if }E_v/F_v\text{ is ramified,}
\end{cases}
\right\}
&\text{if }G=\GL(2),\\[7mm]
\vol(\{\pm 1\}\backslash E_v^{(1)})
=\frac{1}{2}\vol(\OOO_{E_v}^{(1)})
&\text{if }G=\SL(2).
\end{cases}
\end{multline*}
By the computations in \cite[\S 5.9]{Ko05} we have for $\GL(2)$ (the quotient of the volume factors before the integral is necessary to obtain the same normalization as in \cite{Ko05}):
\begin{multline*}
\frac{\vol_{Z(F_v)}(Z(F_v)\cap\cpt_v^F)}{\vol_{G(F_v)}(\cpt_v^F)}
\int_{Z(F_v)\backslash G(F_v)} \One_{\cpt_v^F}(x^{-1}\gamma x)\,dx\\
=\begin{cases}
\frac{q_v^{d_{\gamma}+1}-1}{q_v-1}+\frac{q_v^{d_{\gamma}}-1}{q_v-1}
&\text{if }E_v/F_v\text{is unramified},\\
 2 \frac{q_v^{d_{\gamma}+1}-1}{q_v-1}
&\text{if }E_v/F_v\text{is ramified},
\end{cases}
\end{multline*}
where $d_{\gamma}$ denotes the $F_v$-valuation of $(\tr\gamma)^2-4\det\gamma$ in the case that $E_v/F_v$ is unramified, and is determined as in \cite{Ko05}, again in terms of $\gamma$, if the extension is ramified. Further, $\vol_{Z(F_v)}$ (resp. $\vol_{G(F_v)}$) indicates that the volume is taken with respect to our measure on $Z(F_v)$ (resp. $G(F_v)$). Since
\[
\frac{\vol_{Z(F_v)}(Z(F_v)\cap\cpt_v^F)}{\vol_{G(F_v)}(\cpt_v^F)}
=\begin{cases}
\frac{\vol(\OOO_{F_v}^{\times})}{\vol(\OOO_{F_v})\vol(\OOO_{F_v}^{\times})^2}
=\frac{1}{\vol(\OOO_{F_v})\vol(\OOO_{F_v}^{\times})}
&\text{if }G=\GL(2),\\[4mm]
\frac{2}{\vol(\OOO_{F_v})\vol(\OOO_{F_v}^{\times})}
&\text{if }G=\SL(2),
\end{cases}
\]
multiplying by the inverse of the volume of $Z(F_v)\backslash G_{\gamma}(F_v)$ this proves the assertion in the non-split case for $G=\GL(2)$.

For $G=\SL(2)$ first note that if $\One_{\cpt_f^F}(x^{-1}\gamma x)=\One_{\GL(\OOO_{F_v})}(x^{-1}\gamma x)$ for all $\gamma\in G(F)$ and $x\in G(F_v)$. 
The integral
\[
 \frac{\vol_{Z(F_v)}(Z(F_v)\cap\cpt_v^F)}{\vol_{G(F_v)}(\cpt_v^F)}
\int_{Z(F_v)\backslash G(F_v)} \One_{\cpt_v^F}(x^{-1}\gamma x)\,dx\\
\]
can be computed as in \cite[\S 5.9]{Ko05} in the case of $\GL(2)$ by counting vertices in the same Bruhat-Tits building which are fixed by $\gamma$. However, for $\SL(2)$ not all vertices necessarily represent elements in $G(F_v)/\cpt_v^F$ so that the set of fixed points might has less elements than in the $\GL(2)$ case. Hence 
\begin{multline*}
\frac{\vol_{Z(F_v)}(Z(F_v)\cap\cpt_v^F)}{\vol_{G(F_v)}(\cpt_v^F)}
\int_{Z(F_v)\backslash G(F_v)} \One_{\cpt_v^F}(x^{-1}\gamma x)\,dx\\
\le \begin{cases}
 \frac{q_v^{d_{\gamma}+1}-1}{q_v-1}+\frac{q_v^{d_{\gamma}}-1}{q_v-1}
&\text{if }E_v/F_v\text{is unramified},\\
 2 \frac{q_v^{d_{\gamma}+1}-1}{q_v-1}
&\text{if }E_v/F_v\text{is ramified}.
\end{cases}
\end{multline*}
The estimates for the integrals over $G_\gamma(\A_{F,f})\backslash G(\A_F)$ now follows by multiplying all the local estimates and taking into account that $E_v/F_v$ can be ramified only if the relative different is divisible by $v$. But this can only happen if the residue characteristic of $v$ is even or one of finitely many prime numbers which are determined by $\gamma$ alone. Moreover, the residue characteristic of those $v$ for which $d_\gamma\not=0$, is also contained in a finite set of prime numbers depending only on $\gamma$ but not on $F$. More precisely, the sets of such prime numbers only depend on the characteristic polynomial of $\gamma$.  Since the set of all occurring  characteristic polynomials is finite and independent of $F$, the assertion follows.
\end{proof}

\begin{cor}\label{cor:reg:ell}
There exists a constant $c>0$ such that for all $F\in\F_{\signature}$ we have
\[
 j_{\text{reg.ell.}}^{F,T}(f_F)
\leq
\begin{cases}
c D_F^{-\frac{1}{2}} (\log D_F)^{2d}				&\text{if }G=\GL(2),\\
c (\log D_F)^{2d}								&\text{if }G=\SL(2),
\end{cases}
\]
for all $T\in\aaa^+$ with $\alpha(T)\geq \rho\log D_F$.
\end{cor}

\begin{proof}
We consider the case $G=\GL(2)$ first.
We need to estimate from above the volume of the quotient
\[
G_{\gamma }(F)\backslash G_{\gamma}(\A_F)^1
=F(\gamma)^{\times}\backslash\A_{F(\gamma)}^1,
\]
where $F(\gamma)$ denotes the quadratic splitting field of $\gamma$ over $F$. By our normalization of measures, the volume of this last quotient is
\[
 \vol(F(\gamma)^{\times}\backslash\A_{F(\gamma)}^1)
=\res_{s=1}\zeta_{F(\gamma)}(s),
\]
where $\zeta_{F(\gamma)}(s)$ denotes the Dedekind zeta function for the field $F(\gamma)$. By \eqref{eq:upper:bound:residue} the residuum can be estimated in terms of the discriminant of $F(\gamma)$ so that we need to compute $D_{F(\gamma)}$. For a finite field extension $K_1/K_2$ let $\delta_{K_1/K_2}\subseteq \OOO_{K_1}$ denote the relative different of $K_1$ over $K_2$. Then
\[
D_{F(\gamma)}
=\N_{F(\gamma)/\Q}(\delta_{F(\gamma)/\Q})
=\N_{F(\gamma)/\Q}(\delta_{F(\gamma)/F}) D_F^2
\]
where $\N_{F(\gamma)/\Q}$ denotes the ideal norm of the field extension $F(\gamma)/\Q$. 
Now $F(\gamma)$ is a quadratic extension of $F$ and the characteristic polynomial of $\gamma$ is in a  fixed finite set (independent of $F$) so that $\N_{F(\gamma)/\Q}(\delta_{F(\gamma)/F})\ll_{d, f_\infty} 1$. 
Hence using \eqref{eq:upper:bound:residue} we have 
\begin{equation}\label{eq:bound:volume:centralizer}
\vol(G_{\gamma }(F)\backslash G_{\gamma}(\A_F)^1)
\ll_{d,f_\infty} (\log D_{F(\gamma)})^{2d-1}
\ll_{d,f_\infty} (\log D_F)^{2d-1}.
\end{equation}
By Lemma \ref{lem:orb:int:elliptic} and our normalization of measures we get
\begin{multline}\label{eq:bound:orb:integral:gl2:elliptic}
\int_{G_{\gamma}(\A_{F,f})\backslash G(\A_{F,f})} \One_{\cpt_f^F}(x^{-1}\gamma x)\,dx
\ll_{d, f_\infty} \log D_F D_F^{-3/2} D_{F(\gamma)}^{1/2}\\
=  \log D_F D_F^{-1/2} \N_{F(\gamma)/\Q}(\delta_{F(\gamma)/F})^{1/2}
\ll_{d, f_\infty} \log D_F D_F^{-1/2}.
\end{multline}
Putting \eqref{eq:bound:volume:centralizer} and \eqref{eq:bound:orb:integral:gl2:elliptic} together, taking into account that the archimedean orbital integrals take on values in a finite set only (independent of $F\in\F_{\signature}$) and that $|\Sigma_0(F)_{\text{reg.ell}}|\leq |\Sigma_0|$ (recall that $\GL_2(F)$-conjugacy classes in $\Sigma(F)_{\text{reg.ell.}}$ are in bijection with elements in $\Sigma_0(F)_{\text{reg.ell.}}$),  yields the assertion for $G=\GL(2)$.

Now assume that $G=\SL(2)$. In this case
\[
G_{\gamma}(F)\backslash G_{\gamma}(\A_F)^1
\simeq F(\gamma)^{(1)} \backslash \A_{F(\gamma)}^{(1)},
\]
where $F(\gamma)^{(1)}=\OOO_{F(\gamma)}^{(1)}$ denotes the set of norm $1$ elements in $F(\gamma)$ for the norm $F(\gamma)\ni x\mapsto x\bar{x}\in F$, with $\bar{x}$ denoting the image of $x$ under the non-trivial $F$-linear involution on $F(\gamma)$, and $\A_{F(\gamma)}^{(1)}= F(\gamma)_{\infty}^{(1)}\times\widehat{\OOO_{F(\gamma)}}^{(1)}$ with $F(\gamma)_{\infty}^{(1)}=\OOO_{F(\gamma)}^{(1)}\otimes\R\subseteq F(\gamma)_{\infty}$. Note that if $x\in \A_{F(\gamma)}^\times$ the condition that $x\in\A_{F(\gamma)}^{(1)}$ is equivalent to the condition that $|x|_v:=\prod_{w|v} |x_w|_w=1$ for every place $v$ of $F$. Here $w$ runs over all places of $F(\gamma)$ lying above $v$. Hence we canonical have
\[
 \A_{F(\gamma)}^{(1)}\hookrightarrow \A_{F(\gamma)}^1\twoheadrightarrow F(\gamma)^\times \backslash \A_{F(\gamma)}^1
\]
and the kernel of the composition of the maps equals $F(\gamma)^\times\cap \A_{F(\gamma)}^{(1)}= F(\gamma)^{(1)}$. Hence 
\[
 F(\gamma)^{(1)} \backslash \A_{F(\gamma)}^{(1)}
 \subseteq F(\gamma)^\times\backslash \A_{F(\gamma)}^1,
\]
so that 
\[
\vol\left(G_{\gamma}(F)\backslash G_{\gamma}(\A_F)^1\right)
= \vol\left(F(\gamma)^{(1)} \backslash \A_{F(\gamma)}^{(1)}\right)
\le \vol\left(F(\gamma)^\times\backslash \A_{F(\gamma)}^1\right)
=\res_{s=1}\zeta_{F(\gamma)}(s).
\]
Using the bound \eqref{eq:upper:bound:residue} for the residuum
 and combining this with Lemma \ref{lem:orb:int:elliptic} we obtain
\begin{align*}
\vol(G_{\gamma}(F)\backslash G_{\gamma}(\A_F)^1)
&\int_{G_{\gamma}(\A_{F,f})\backslash G(\A_{F,f})} \One_{\cpt_f^F}(x^{-1}\gamma x)\,dx\\
&\ll_{d, f_\infty} (\log D_{F(\gamma)})^{2(d-1)} \Delta_{\signature}(\gamma) \vol(\widehat{\OOO_F})\vol(\widehat{\OOO_F}^{\times}) \\
& \ll_d (\log D_{F(\gamma)})^{2(d-1)}   D_F^{-1} \log D_F  
\end{align*}
Since $D_{F(\gamma)}$ can be bounded by a constant multiple of $D_F$ which depends on the characteristic polynomial of $\gamma$ alone, and since the set of occuring characteristic polynomials is finite and only depending on $d$ and $f_\infty$, we obtain
\[
\vol(G_{\gamma}(F)\backslash G_{\gamma}(\A_F)^1)
\int_{G_{\gamma}(\A_{F})\backslash G(\A_{F})} f_{\infty}\cdot\One_{\cpt_f^F}(x^{-1}\gamma x)\,dx
\ll_{d,f_\infty} (1+ \log D_F)^{2d}.
\]
The assertion for $G=\SL(2)$ then follows as for $\GL(2)$ with the difference that number of $\SL_2(F)$-conjugacy classes in $\Sigma(F)_{\text{reg.ell.}}$ is bounded by $c|\Sigma_0(F)_{\text{reg.ell}}\le c|\Sigma_0|$ with $c$ as in Lemma~\ref{lem:classes}.
\end{proof}

\section{The regular split contribution}
In this section we bound the contribution of regular split conjugacy classes in \eqref{eq:splitting}.
If $\gamma\in \Sigma(F)_{\text{reg.split}}$, then it is conjugate in $G(F)$ to a diagonal element $\diag(\gamma_1, \gamma_2)\in G(F)$, $\gamma_1\neq\gamma_2$. Let $\Sigma'_0(F)_{\text{reg.split}}$ be a set of diagonal representatives for the $G(F)$-conjugacy classes in $\Sigma(F)_{\text{reg.split}}$, that is, for every $\gamma\in\Sigma_0(F)_{\text{reg.split}}$ there is exactly one $\delta\in\Sigma_0'(F)_{\text{reg.split}}$ such that $\delta$ and $\gamma$ are conjugate in $G(F)$.
We can write $j_{\text{reg.split}}^{F,T}(f_F)$  as
\begin{align*}
 j_{\text{reg.split}}^{F,T}(f_F)
& =\sum_{\delta\in \Sigma_0'(F)_{\text{reg.split}}} \int_{G_{\delta}(F)\backslash G(\A_F)} F(x,T) f_F(x^{-1}\delta x)\,dx \\
& = \sum_{\delta\in \Sigma_0'(F)_{\text{reg.split}}}\int_{U_0(\A_F)} f_F(u^{-1}\delta u) \int_{A_G T_0(F)\backslash T_0(\A_F)} F(tu,T)\,dt \,du,
\end{align*}
where we used $G_{\delta}(F)=T_0(F)$ for the second equality.

\begin{proposition}\label{prop:split:contribution}
 There exists a constant $c>0$ such that for all $F\in\F_{\signature}$ and all  $T\in\aaa^+$ with $\alpha(T)\geq\rho\log D_F$ we have
\[
 \frac{j_{\text{reg.split}}^{F,T}(f_F)}{\vol(\widehat{\OOO_F}) \vol(T_0(F)\backslash T_0(\A_F)^1) } 
\leq c\varpi(T).
\]
\end{proposition}

For the proof of this proposition we need the following lemma:
\begin{lemma}\label{lem:torus:int:truncation}
For any $u=\left(\begin{smallmatrix}1&x\\0&1\end{smallmatrix}\right)\in U_0(\A_F)$ and any $T\in\aaa^+$ with $\alpha(T)\geq\rho\log D_F$ we have
\begin{equation}\label{eq:trunc:torus}
\frac{\int_{A_G T_0(F)\backslash T_0(\A_F)} F(tu,T)\, dt}{\vol(T_0(F)\backslash T_0(\A_F)^1)}
\leq \delta_G \left(\varpi(T)+\log\|(1,x)\|_{\A_F}\right)
\end{equation}
where $\delta_G=1$ if $G=\SL(2)$ and $\delta_G=2$ if $G=\GL(2)$.
\end{lemma}
\begin{proof}
If $G=\SL(2)$, the integral on the left hand side of \eqref{eq:trunc:torus} equals
\[
 \int_{F^{\times}\backslash\A_F^\times} F(\left(\begin{smallmatrix}a& ax\\0&a^{-1}\end{smallmatrix}\right), T)\, d^\times a.
\]
On the other hand, for $G=\GL(2)$, the left hand side of \eqref{eq:trunc:torus} equals
\[
 \vol(F^{\times}\backslash\A_F^1)\int_{F^\times\backslash \A_F^\times} F(\left(\begin{smallmatrix}a& ax\\0&1\end{smallmatrix}\right), T)\, d^\times a
\]
since $F(\cdot, T)$ is invariant under the center of the group.

Let $w=\left(\begin{smallmatrix}0&1\\-1&0\end{smallmatrix}\right)$. Then $F(g, T)=1$ implies that 
\[
 \langle \varpi,H_0(g)-T\rangle\leq 0 \text{ and } \langle \varpi,H_0(wg)-T\rangle\leq 0.
\]
Hence $F(\left(\begin{smallmatrix}a& ax\\0&a^{-1}\end{smallmatrix}\right), T)=1$ implies that
\[
 e^{\varpi(T)}   \ge |a|_{\A_F}^2 \ge  \|(1,x)\|_{\A_F}^{-2}e^{-\varpi(T)}
\]
while $F(\left(\begin{smallmatrix}a& ax\\0&1\end{smallmatrix}\right), T)=1$ implies that
\[
 e^{\varpi(T)}  \ge |a|_{\A_F}\ge  \|(1,x)\|_{\A_F}^{-2} e^{-\varpi(T)}.
\]
Here $\|(1,x)\|_{\A_F}$ is defined as $\|(1,x)\|_{\A_F}=\prod_{v}\|(1,x_v)\|_{v}$ with
\[
\|(1,x)\|_v
=
\begin{cases}
\max\{1,|x|_v\}
&\text{if } v\text{ is non-archimedean},\\
\sqrt{1+x^2}
&\text{if }v=\R,\\
1+x\bar x
&\text{if }v=\C.
\end{cases}
\]
 Note that $\|(1,x)\|_{\A_F}\geq1$ for all $x$.
Hence  the left hand side of \eqref{eq:trunc:torus} is bounded from above by
\[
 \vol(F^{\times}\backslash\A_F^1) \int_{e^{-\varpi(T)/2}  \|(1,x)\|_{\A_F}^{-1} }^{ e^{\varpi(T)/2}}\, d^\times a
 = \vol(F^{\times}\backslash\A_F^1) \left( \varpi(T)+\log \|(1,x)\|_{\A_F}\right)
\]
if $G=\SL(2)$, and by
\[
  \vol(F^{\times}\backslash\A_F^1)^2 \int_{e^{-\varpi(T)} \|(1,x)\|_{\A_F}^{-2} }^{ e^{\varpi(T)} }\, d^\times a
  = \vol(F^{\times}\backslash\A_F^1)^2 \left(2\varpi(T) + 2\log \|(1,x)\|_{\A_F}\right)
\]
if $G=\GL(2)$. 
This proves the lemma.
\end{proof}

\begin{proof}[Proof of Proposition \ref{prop:split:contribution}]
We use the notation from the beginning of this section. After a change of variables $(\gamma_1-\gamma_2)x\mapsto x$ we get by Lemma \ref{lem:torus:int:truncation} 
that
\begin{multline*}
 \vol(T_0(F)\backslash T_0(\A_F)^1)^{-1}j_{\text{reg.split}}^{F,T}(f_F)\\
\leq  2\varpi(T)\sum_{\delta\in \Sigma_0'(F)_{\text{reg.split}}}
\int_{\A_F} f_F(\left(\begin{smallmatrix}\gamma_1&x\\0&\gamma_2\end{smallmatrix}\right))
\left(1+ \sum_{v\leq \infty}
\Big|\log\|(|\gamma_1-\gamma_2|_v,|x|_v)\|_v\Big| \right)\, dx.
\end{multline*}
For $\delta\in\Sigma_0'(F)_{\text{reg.split}}$ let $S_{\delta}$ be the finite set of places $v$ of $F$ with $|\gamma_1-\gamma_2|_v\neq1$ or $v$ archimedean. Let $S_{\delta,f}$ denote the set of all non-archimedean places contained in $S_{\delta}$. Then for every $v\not\in S_{\delta}$ we have
\[
 \int_{F_v}\One_{\cpt_v^F}( \left(\begin{smallmatrix}\gamma_1&x\\0&\gamma_2\end{smallmatrix}\right))
 \Big|\log\|(|\gamma_1-\gamma_2|_v,|x|_v)\|_v\Big| \,dx
=0.
 \]
If $v\in S_{\delta}$ is non-archimedean, then
\[
 \int_{F_v}\One_{\cpt_v^F}( \left(\begin{smallmatrix}\gamma_1&x\\0&\gamma_2\end{smallmatrix}\right))
 \Big|\log\|(|\gamma_1-\gamma_2|_v,|x|_v)\|_v\Big| \,dx
\leq \Big|\log|\gamma_1-\gamma_2|_v \Big| \vol(\OOO_{F_v}).
\]
Note that
\[
\sum_{v\in S_{\delta,f}}\Big|\log|\gamma_1-\gamma_2|_v\Big|
=\Big|\log|\gamma_1-\gamma_2|_{\A_{F,f}}\Big|
=\log|\gamma_1-\gamma_2|_{\signature}
\]
since $\gamma_1, \gamma_2\in\OOO_F$. 

Further, if $v$ is non-archimedean, we have
\[
\int_{F_v}\One_{\cpt_v^F}( \left(\begin{smallmatrix}\gamma_1&x\\0&\gamma_2\end{smallmatrix}\right))\,dx
=\begin{cases}
 \vol(\OOO_{F_v})				&\text{if } \gamma_1\gamma_2=\det\delta \in\OOO_{F_v}^\times,\\
 0						&\text{else.}
 \end{cases}
\]
Note that $(\gamma_1-\gamma_2)^2=(\tr\delta)^2-4\det\delta$ so that the terms $|\gamma_1-\gamma_2|_v$ for $v$ archimedean depend only on the set $\Sigma_0(F)_{\text{reg.split}}$ but not on the diagonal conjugacy class representatives.
Hence
\[
\frac{j_{\text{reg.split}}^{F,T}(f_F)}{\vol(\widehat{\OOO_F})  \vol(T_0(F)\backslash T_0(\A_F)^1)}\\
\]
 is bounded from above by 
 \begin{multline}\label{eq:final:reg:split}
c \varpi(T)  \sum_{\delta\in \Sigma_0'(F)_{\text{reg.split}}}  \left(1+\log|\gamma_1-\gamma_2|_{\signature}\right)\\
\cdot\int_{U_0(\R^{\signature})} f_{\infty}(\left(\begin{smallmatrix}\gamma_1&x\\0&\gamma_2\end{smallmatrix}\right))
\left(1+\sum_{v|\infty}\Big|\log\|(|\gamma_1-\gamma_2|_v,|x|_v)\|_v\Big|\right)\,dx
\end{multline}
for some constant $c>0$ independent of $F$ and $T$. Now all appearing quantities in this last sum and integral depend only on the polynomials in $\Sigma_0(F)_{\text{reg.split}}$, that is, only on the polynomials in this set, but not on the specific representatives for the attached conjugacy classes over $F$. Since $\Sigma_0(F)_{\text{reg.split}}$ is contained in the finite set $\Sigma_0$ (which is independent of $F$) and there are at most two $G(F)$-conjugacy classes in each regular split equivalence class, the right hand side in \eqref{eq:final:reg:split} can be bounded by a $C\varpi(T)$ with $C>0$ an absolute constant.
\end{proof}

\section{The unipotent contribution}\label{sec:unip:contr}
In this section we bound the final contribution in \eqref{eq:splitting}, namely the contribution from elements of the form $zu$ with $z$ central and $u\neq 1$ unipotent. More precisely we have the following:
\begin{proposition}\label{prop:unip:contr}
There exists a constant $c>0$ such that for all $F\in\F_{\signature}$ and all $T\in\aaa^+$ with $\alpha(T)\geq\rho\log D_F$ we have
\[
\frac{ j_{\text{unip}\smallsetminus Z}^{F,T}(f_F)}{\vol(T_0(F)\backslash T_0(\A_F)^1)  \vol(F\backslash\A_F)}
\leq c  \varpi(T).
\]
\end{proposition}

We need some notation and auxiliary results before proving this lemma.

First note that there exist finitely many $z_1, \ldots, z_t\in \bar\Q$ such that for any $F\in\F_{\signature}$ we have 
\[
 \Sigma_0(F)_{\text{unip}}=\begin{cases}
                            \{(X-z)^2\mid z\in \{z_1, \ldots, z_t\}\cap F\}			&\text{if } G=\GL(2),\\
                            \{(X-z)^2\mid z\in \{z_1,\ldots, z_t\}\cap \{1, -1\}\}		&\text{if } G=\SL(2).
                           \end{cases}
\]
Accordingly,
\[
 \Sigma(F)_{\text{unip}}=\{ z u(x)\mid z\in \{z_1, \ldots, z_t\}\cap F, \, x\in F\}
\]
if $G=\GL(2)$, where $u(x)=\left(\begin{smallmatrix} 1&x\\0&1\end{smallmatrix}\right)$, and 
\[
 \Sigma(F)_{\text{unip}}=\{ z u(x)\mid z\in \{z_1, \ldots, z_t\}\cap \{1, -1\}, \, x\in F\}
\]
if $G=\SL(2)$.

If $x\neq0$, the centralizer $G_{zu(x)}$ of $zu(x)$ in $G$ equals $ZU_0$. For $a\in\A_F^{\times}$ we set 
\[
t(a)=
\begin{cases}
\diag(a,1)					&\text{if }G=\GL(2),\\
\diag(a, a^{-1}) 				&\text{if }G=\SL(2).
\end{cases}
\]
Write $\nu(Z):=\vol(Z(F)\backslash Z(\A_F)^1)$. 
We can then write the contribution from the unipotent but non-central elements as
\begin{align*}
& j_{\text{unip}\smallsetminus Z}^{F,T}(f_F)\\
& = \sum_{z} \int_{T_0(F)A_G\backslash T_0(\A_F)}\delta_0(t)^{-1}
 \sum_{x\in F^{\times}} f_F(z t^{-1}u(x)t)
\int_{U_0(F)\backslash U_0(\A_F)} F(vt,T) \,dv\,dt\\
& = \sum_{z} \nu(Z) \int_{F^{\times}\backslash\A_F^{\times}} \delta_0(t)^{-1}
\sum_{x\in F^{\times}} f_F(z t(a)^{-1}u(x)t(a))
\int_{U_0(F)\backslash U_0(\A_F)} F(vt(a),T) \,dv\,d^{\times}a,
\end{align*}
where the sum runs over all $z\in\{z_1,\ldots, z_t\}\cap F$ if $G=\GL(2)$ and over $z\in\{z_1,\ldots, z_t\}\cap\{1, -1\}$ if $G=\SL(2)$. Since the sum over the $z$ is finite, we can ignore $z$ in the following and just find an upper bound for
\begin{equation}\label{eq:int:unip:summand}
 \int_{F^{\times}\backslash\A_F^{\times}} \delta_0(t)^{-1}
\sum_{x\in F^{\times}} f_F(t(a)^{-1}u(x)t(a))
\int_{U_0(F)\backslash U_0(\A_F)} F(vt(a),T) \,dv\,d^{\times}a.
\end{equation}

Suppose $a=a_\infty a_f\in\A_F^\times$, $a_\infty\in F_\infty^\times$, $a_f\in\A_{F,f}^\times$. Let $\phi:\A_F\longrightarrow\infty$ be a Schwartz Bruhat function such that $\phi=\phi_{\infty}\cdot\phi_f$ with $\phi_f:\A_{F,f}\longrightarrow\C$ the characteristic function of $\widehat{\OOO_F}=\bigotimes_{v<\infty} \OOO_{F_v}$, and $\phi_\infty$ a smooth function on $F_{\infty}=\R^{\signature}$ with support in a ball $\{x\in\R^{\signature}\mid\|x\|_{\signature}\le R\}$. Consider the sum
\[
 \sum_{x\in F^{\times}} \phi(ax).
\]
We want to find an upper bound for it so that we can apply it to $\phi(x)=f_F(u(x))$. For $\phi(ax)$ to be non-zero we need that $a_fx_f\in \widehat{\OOO_F}$, that is, $x_f\in a_f^{-1}\widehat{\OOO_F}\cap F=:\Lambda_a'$. $\Lambda_a'$ is a fractional ideal in $F$ and by replacing $a$ by some suitable element in $a F^{\times}$ we can assume that $\Lambda_a'$ is an integral ideal in $\OOO_F$. The norm of this ideal is 
\[
 \N(\Lambda_a')= \left|\OOO_F/\Lambda_a'\right| = |a_f|_f.
\]
Put $\Lambda_a:=a_{\infty}\Lambda_a'$. 
This is a lattice in $\R^{\signature}$. $\|\cdot\|_{\signature}^2$ defines a positive definite quadratic form on $\R^{\signature}$. Let $\lambda_1(\Lambda_a)$ denote the first successive minimum of the lattice $\Lambda_a$ with respect to the form $\|\cdot\|_{\signature}^2$.

\begin{lemma}\label{lem:first:minimum}
 For any $a$ as above we have
 \[
  \lambda_1(\Lambda_a)\ge |a|_{\A_F}^{1/d}.
 \]
\end{lemma}
\begin{proof}
 By the arithmetic geometric mean inequality we have for any $y\in\Lambda_a$, $y\neq 0$ that
 \[
  \|y\|_{\signature}^2\ge |y|_\infty^{1/d}.
   \]
Write $y=a_\infty y'$ with $y'\in\Lambda_a'$, $y'\neq0$. Then 
\[
|y|_\infty^{1/d}
 = |a_\infty|_\infty^{1/d} |y'|_\infty^{1/d} 
 = |a_\infty|_\infty^{1/d} |y'|_f^{-1/d} 
 \ge |a_\infty|_\infty^{1/d} |a_f|_f^{1/d} 
 = |a|_{\A_F}^{1/d}
\]
so that $\lambda_1(\Lambda_a)\ge |a|_{\A_F}^{1/d}$ as claimed.
\end{proof}

\begin{lemma}\label{lem:number:lattice:pts}
Let  $R>0$ be fixed.
There exist constants $c, C>0$ depending only on $d$ and $R$ such that for every $F\in\F_{\signature}$ and every $a\in F^{\times}\backslash \A_F^{\times}$ we have
\begin{equation}\label{eq:number:lattice:pts}
\left|\{X\in \Lambda_a \smallsetminus\{0\} \mid \|X\|_{\signature}\leq R\}\right|
\begin{cases}
=0
&\text{if } |a|_{\A_F} >C,\\
\leq c |a|_{\A_F}^{-1}
&\text{else.}
\end{cases}
\end{equation}
\end{lemma}

\begin{proof}
We use \cite[Theorem 2.1]{BeHeWi93} to estimate the number of points in the set in \eqref{eq:number:lattice:pts}. We obtain
\[
\left|\{X\in \Lambda_a\smallsetminus\{0\} \mid \|X\|_{\signature}\leq R\}\right|
\begin{cases}
=0
&\text{if }\lambda_1(\Lambda_a)> R^2,\\
\leq \left(\left\lfloor \frac{2}{R^{-2}\lambda_1(\Lambda_a)}+1\right\rfloor\right)^d
&\text{if }\lambda_1(\Lambda_a)\leq R^2.
\end{cases}
\]
The case $\lambda_1(\Lambda_a)\leq R^2$ can be bounded by
\[
(2R)^{2d} \lambda_1(\Lambda_a)^{-d}
\]
By Lemma \ref{lem:first:minimum} we have $\lambda_1(\Lambda_a)\ge |a|_{\A_F}^{1/d}$. Hence the left hand side of \eqref{eq:number:lattice:pts} is empty if $|a|_{\A_F}> R^{2d}$, and bounded by $(2R)^{2d} |a|_{\A_F}^{-1}$ in all other cases.
\end{proof}

\begin{cor}
There exist constants $c,C>0$ depending only on $d$ and $f_{\infty}$ such that
\[
\sum_{x\in F^\times} f_F(t(b)^{-1} u(x) t(b))
\begin{cases}
=0
&\text{if }|b|_{\A_F}^{\beta_G}< C,\\
\leq c |b|_{\A_F}^{\beta_G}
&\text{else,}
\end{cases}
\]
for any $b\in \A_F^\times$. Here $\beta_{\SL(2)}=2$  and $\beta_{\GL(2)}=1$. 
\end{cor}
\begin{proof}
Note that 
\[
\sum_{x\in F^\times} f_F(t(b)^{-1} u(x) t(b))
=\sum_{x\in \Lambda_{b^{-1}}'} f_{\infty} (u(b_\infty^{-\beta_G} x) ).
\]
Recall that $f_{\infty}$ is supported in the compact set
$\{A=(A_{ij})_{i,j}\in\Mat_{2\times 2}(\R^{\signature})\mid\forall i,j:~\|A_{ij}-\delta_{ij}\|_{\signature}\leq R\}$,
where $\delta_{ij}$ denotes the Kronecker delta. Hence we can apply Lemma \ref{lem:number:lattice:pts} with $a=b^{-2}$ if $G=\SL(2)$ and with $a=b^{-1}$ if $G=\GL(2)$. 
\end{proof}

Note that with $\beta_G$ as in the corollary we have $|b|_{\A_F}^{\beta_G}=\delta_0(t(b))$. Therefore this corollary gives a constant $c>0$ such that \eqref{eq:int:unip:summand} is bounded from above by a constant multiple of 
\[
\int_{c}^{\infty} \int_{F^\times\backslash\A_F^1}\int_{U_0(F)\backslash U_0(\A_F)} F(vt(ab),T) \,dv\,d^\times b\, d^{\times}a
\]
Now, similarly as in the regular split case, the condition $F(v t(ab), T)=1$ gives upper and lower bounds on $a$ and $b$ in terms of $v$. More precisely, writing $v=u(y)$ with $y\in F\backslash \A_F$, we must have 
\[
 e^{\varpi(T)}  \ge |ab|_{\A_F}^2 \ge \|(1, (ab)^{-2} y)\|_{\A_F}^{-1}  e^{-\varpi(T)}
\]
if $G=\SL(2)$, and 
\[
 e^{\varpi(T)}  \ge |ab|_{\A_F} \ge \|(1,(ab)^{-1}y)\|_{\A_F}^{-1} e^{-\varpi(T)}. 
\]
if $G=\GL(2)$. In particular, $e^{-\varpi(T)}\le |ab|_{\A_f}^m\le e^{\varpi(T)}$ in both cases with $m$ as in above corollary.
Hence \eqref{eq:int:unip:summand} is bounded from above by a constant multiple of 
\[
\vol(F^\times\backslash\A_F^1)\vol(F\backslash \A_F)  \int_{c}^{e^{\varpi(T)}} \, d^{\times}a
\le \vol(F^\times\backslash\A_F^1)\vol(F\backslash \A_F) c_2 \varpi(T)
\]
for some constant $c_2>0$ depending only on $d$ and $f_{\infty}$. 
Finally using $\nu(Z)\vol(F^\times\backslash \A_F^1)=\vol(T_0(F)\backslash T_0(\A_F)^1)$ the proof of Proposition \ref{prop:unip:contr} is finished.

\section{Proof of Lemma \ref{lem:upper:bound:T} and Proposition \ref{prop:upper:bound}}\label{sec:proof}

\begin{proof}[Proof of Lemma \ref{lem:upper:bound:T}]
By Proposition \ref{prop:split:contribution} and Lemma \ref{lem:quot:meas} we have  
\[
\frac{j_{\text{reg.split}}^{F,T}(f_F)}{\vol(G(F)\backslash G(\A_F)^1)}
\ll_{d, f_\infty}  D_F^{-1} (\log D_F)^{d-1} \varpi(T)
\]
for all $F\in\F_{\signature}$  and  all $T\in\aaa^+$ with $\alpha(T)\geq \rho\log D_F$.

For the non-central unipotent contribution we similarly find by Proposition \ref{prop:unip:contr} and Lemma \ref{lem:quot:meas} that
\[
\frac{ j_{\text{unip}\smallsetminus Z}^{F,T}(f_F)}{\vol(G(F)\backslash G(\A_F)^1)}
\ll_{d, f_\infty} D_F^{-1/2} (\log D_F)^{d-1} \varpi(T).
\]

For the contribution from the regular elliptic part consider first $G=\GL(2)$: We bound the volume of $G(F)\backslash G(\A_F)^1$ from below as follows:
By our normalization of measures and the class number formula we have
\begin{multline*}
\vol(G(F)\backslash G(\A_F)^1)
= D_F^{1/2} \res_{s=1} \zeta_F(s) \zeta_F(2)
= D_F^{1/2}\zeta_F(2) \frac{2^{r_1}(2\pi)^{r_2} h_F R_F}{w_F D_F^{1/2}}\\
\geq \frac{2^{r_1}(2\pi)^{r_2} h_F R_F}{w_F }.
\end{multline*}
Using the lower bound for the regulator \eqref{eq:lower:bound:regulator} we have $ \vol(G(F)\backslash G(\A_F)^1)\gg_d 1$. 
Hence combining with Corollary \ref{cor:reg:ell}, we get 
\[
\frac{ j_{\text{reg.ell.}}^{F,T}(f_F)}{\vol(G(F)\backslash G(\A_F)^1)}
\ll_{d,f_\infty} D_F^{-\frac{1}{2}} (\log D_F)^{2d}
\]
for all $F\in\F_{\signature}$, and 
 all $T\in\aaa^+$ with $\alpha(T)\geq\rho\log D_F$.
This finishes the proof of estimate \eqref{eq:upper:bound:T} for $G=\GL(2)$.

For $G=\SL(2)$ we have
\[
\vol(G(F)\backslash G(\A_F)^1)= D_F^{\frac{1}{2}}\zeta_F(2) \geq D_F^{\frac{1}{2}}.
\]
Together with Corollary \ref{cor:reg:ell} we get 
\[
\frac{ j_{\text{reg.ell.}}^{F,T}(f_F)}{\vol(G(F)\backslash G(\A_F)^1)}
\ll_{d, f_\infty} D_F^{-\frac{1}{2}} (\log D_F)^{2d}
\]
for all $F\in\F_{\signature}$. This finishes the proof of the assertion for $G=\SL(2)$.
\end{proof}

Recall  that $\Jgeom^F(f_F)$ as well as $J_{\text{geom}\smallsetminus Z}^F(f_F)$ are the values at $T=0$, i.e., the constant terms, of certain polynomials $\Jgeom^{F,T}(f_F)$ and $J_{\text{geom}\smallsetminus Z}^{F,T}(f_F)$ of degree $1$ (see, for example, \cite[\S 9]{Ar05}). 

The polynomials can  be approximated by $j^{F,T}(f_F)$ and $j_{G\smallsetminus Z}^{F,T}(f_F)$ as follows:
\begin{lemma}\label{lem:approx}
 There exist constants $c_1,c_2, \eps >0$ depending only on $f_{\infty}$ and $d$ such that for all $F\in\F_{\signature}$ we have
\[
 \left|\Jgeom^{F,T}(f_F)-j^{F,T}(f_F)\right|
\leq c_1 D_F^{-1} e^{-\alpha(T)/2} \varpi(T)
\]
and
\[
  \left|J_{\text{geom}\smallsetminus Z}^{F,T}(f_F)-j_{G\smallsetminus Z}^{F,T}(f_F)\right|
\leq c_1 D_F^{-1} e^{-\eps \alpha(T)} \varpi(T)
\]
for all $T\in\aaa^+$ with $\alpha(T)\geq c_2\log D_F$.
\end{lemma}
\begin{proof}
 By \cite[Lemma 7.7]{coeff_est} there are constants $a_1, a_2>0$ depending only on $f_{\infty}$ and $d$ such that
\[
 \left|\Jgeom^{F,T}(f_F)-j^{F,T}(f_F)\right|
\leq a_1 D_F^{a_2} e^{-\alpha(T)} \varpi(T)
\]
for all $T\in\aaa^+$ with $\alpha(T)\geq \rho\log D_F$.\footnote{In \cite[Lemma 7.7]{coeff_est} only the case of the unipotent distributions $J_{\text{unip}}^{F,T}$ and $j_{\text{unip}}^{F,T}$ was considered. It is, however, clear from the proof of \cite[Lemma 7.7]{coeff_est} that the statement of that lemma remains true for $\Jgeom^{F,T}$ and $j_{\text{geom}}^{F,T}$, cf. also \cite[Theorem 1]{Ar79} and \cite[Proposition 4.4]{Ch02}.}
Then $D_F^{a_2} e^{-\alpha(T)/2}\leq 1$ for every $T\in\aaa^+$ with $\alpha(T)\geq 2(a_2+1)\log D_F$. Hence
\[
 \left|\Jgeom^{F,T}(f_F)-j^{F,T}(f_F)\right|
\leq a_1 D_F^{-1} e^{-\alpha(T)/2} \varpi(T) ,
\]
for every $T\in\aaa^+$ with $\alpha(T)\geq\max\{2(a_2+1)\log D_F,\rho\log D_F\}$.

The other inequality follows similarly: Combining the proof of \cite[Lemma 7.7]{coeff_est} with the proofs of \cite[Theorem 1]{Ar79} and \cite[Proposition 4.2]{Ar85} (cf. also \cite[Theorem~3.4]{FiLaMu_limitmult}) we find $a_1, a_2,\delta, \rho'>0$ independent of $F$ such that
\[
   \left|J_{\text{geom}\smallsetminus Z}^{F,T}(f_F)-j_{G\smallsetminus Z}^{F,T}(f_F)\right|
\leq a_1 D_F^{a_2} e^{-\delta \alpha(T)} \varpi(T)
\]
for every $T\in\aaa^+$ with $\alpha(T)\geq\max\{\rho\log D_F, \rho'\log D_F\}$.
Hence
\[
  \left|J_{\text{geom}\smallsetminus Z}^{F,T}(f_F)-j_{G\smallsetminus Z}^{F,T}(f_F)\right|
\leq a_1 D_F^{-1} e^{-\delta \alpha(T)/2} \varpi(T)
\]
for every $T\in\aaa^+$ with $\alpha(T)\geq\max\{\rho\log D_F, \rho'\log D_F, 2(a_2+1) \log D_F\}$.
\end{proof}

\begin{proof}[Proof of Proposition \ref{prop:upper:bound}]
By Lemma \ref{lem:upper:bound:T} and Lemma \ref{lem:approx} there are constants $C_1, C_2>0$ such that for all $F\in\F_{\signature}$ we have
\[
 \frac{J_{\text{geom}\smallsetminus Z}^{F,T}(f_F)}{\vol(G(F)\backslash G(\A_F)^1)}
\leq C_1 D_F^{-\frac{1}{2}} (\log D_F)^{2d} \varpi(T)
\]
for all $T\in\aaa^+$ with $\alpha(T)\geq C_2\log D_F$.
We can therefore deduce  an upper bound for $J_{\text{geom}\smallsetminus Z}^F(f_F)$ from the upper bound for $J_{\text{geom}\smallsetminus Z}^{F,T}(f_F)$ by interpolation. 
Hence there is a constant $a>0$ such that
\[
 \frac{J_{\text{geom}\smallsetminus Z}^F(f_F)}{\vol(G(F)\backslash G(\A_F)^1)}
\leq a D_F^{-\frac{1}{2}} (\log D_F)^{2d +1 }  
\]
for all $F\in\F_{\signature}$. 
This proves the assertion of the proposition.
\end{proof}

\section{Spectral limit multiplicity property}\label{sec:spec:limit}

The purpose of this section is to prove the  spectral limit multiplicity property in the following form:
\begin{proposition}\label{prop:spectral:limit}
Let $J_{\text{spec}}^F$ denote the spectral side of Arthur's trace formula for $G$ over $F$, and $J_{\text{disc}}^F$ the contribution from the representations occurring in $L^2_{\text{disc}}(G(F)\backslash G(\A_F)^1)$ to $J_{\text{spec}}^F$, that is, the trace of the right regular representation restricted to $L^2_{\text{disc}}(G(F)\backslash G(\A_F)^1)$. 
Then for every $f_{\infty}\in \HHH(G(\R^{\signature})^1)$ we have
\[
\lim_{F\in\F_{\signature}}\frac{J_{\text{spec}}^F(f_{\infty}\cdot\One_{\cpt_f^F})
-J_{\text{disc}}^F(f_{\infty}\cdot\One_{\cpt_f^F})}{\vol(G(F)\backslash G(\A_F)^1)}
=0.
\]
More precisely, for any $\eps>0$ we have 
\[
\left|\frac{J_{\text{spec}}^F(f_{\infty}\cdot\One_{\cpt_f^F})
-J_{\text{disc}}^F(f_{\infty}\cdot\One_{\cpt_f^F})}{\vol(G(F)\backslash G(\A_F)^1)}\right|
\ll_{\signature, \eps} \nu_F^{-\delta_G +\eps} 
\]
with $\delta_G=1/2$ if $G=\GL(2)$ and $\delta_G=2/3$ if $G=\SL(2)$.
\end{proposition}

To prove this proposition, we write the remaining part of the spectral side as follows:
For any $f\in C_c^{\infty}(G(\A_F)^1)$ we have
\[
J_{\text{spec}}^F(f) - J_{\text{disc}}^F(f)
= -\frac{1}{4\pi}\int_{i\R} \tr\left(M(\lambda)^{-1}M'(\lambda) \rho(\lambda, f)\right) \,d\lambda
+\frac{1}{4} \tr\left(M(0)\rho(0, f)\right),
\]
see \cite{GeJa79,Gel96}.
Here the notation is as follows:
\begin{itemize}
\item Let $\AAA$ be the Hilbert space completion of the vector space of all smooth functions $\varphi: T_0(F) U_0(\A_F)\backslash G(\A_F)^1\longrightarrow\C$ such that $\varphi(ag)=\delta_0(a)^{1/2}\varphi(g)$ for all $a\in A_0$ and $g\in G(\A_F)^1$ which are square-integrable over the quotient $T_0(F) U_0(\A_F)\backslash G(\A_F)^1$.
Then $\rho(\lambda, \cdot):\AAA\longrightarrow\AAA$ denotes the induced representation given by
\[
(\rho(\lambda, g)\varphi)(x)
=\varphi(gx) e^{\langle \lambda, H_0(gx)-H_0(x)\rangle},\;\;
g\in G(\A_F)^1,\;\varphi\in\AAA.
\]
\item
$M(\lambda)$ denotes the intertwining operator attached to this induced representation $\rho(\lambda,\cdot)$ via Eisenstein series.
\end{itemize}
$\AAA$ decomposes as $\widehat\bigoplus_{\chi\in\Pi_{\text{disc}}(T_0(\A_F)^1)} \AAA_{\chi}$ where
\begin{itemize}
\item
$\chi=(\chi_1,\chi_2)$ runs over all discrete representations of $T_0(\A_F)^1$, that is, all pairs of unitary characters $\chi_1, \chi_2:F^{\times}\backslash \A_F^1\longrightarrow\C$. If $G=\SL(2)$, $\chi_1$ and $\chi_2$ satisfy $\chi_2=\chi_1^{-1}$.
\item
$\AAA_{\chi}$ denotes the subspace of $\AAA$ of all $\varphi\in \AAA$ which transform according to $\chi$, that is,  $\varphi(\diag(t_1,t_2)g)=\chi_1(t_1)\chi_2(t_2)\varphi(g)$ for all $\diag(t_1, t_2)\in T_0(\A_F)^1$.
\end{itemize}
Let $\rho_{\chi}(\lambda,\cdot)$ denote the restriction of $\rho(\lambda, \cdot)$ to $\AAA_{\chi}$.
Since our test function $f_F=f_{\infty}\cdot\One_{\cpt_f^F}$ is $\cpt_f^F$-invariant, we have $\rho_{\chi}(\lambda, f_F)=0$ unless $\chi_1$ and $\chi_2$ are unramified at all finite places. 

\begin{lemma}\label{lem:spec:term1}
There exists a constant $c>0$ depending only on $f_{\infty}$ and $d$ such that for all $F\in\F_{\signature}$ we have
\[
|\tr M(0)\rho(0, f_F)|
\leq
 c h_F^{a_G} \vol_{G(\A_{F,f})}(\cpt_{f}^F)
\]
where $a_G=2$ if $G=\GL(2)$, and $a_G=1$ if $G=\SL(2)$.
\end{lemma}
\begin{proof}
$M(0)$ is a unitary operator so that it suffices to estimate
\[
|\tr\rho(0,f_F)|
\leq \sum_{\chi} |\tr\rho_{\chi}(0, f_F)|,
\]
where the sum runs over pairs of unramified characters $\chi=(\chi_1, \chi_2)$ with $\chi_2=\chi_1^{-1}$ in the case of $G=\SL(2)$. The number of unramified characters $F^{\times}\backslash \A_F^1\longrightarrow\C$ equals the class number $h_F$, so that the number of summands equals $h_F^2$ if $G=\GL(2)$ and $h_F$ if $G=\SL(2)$.
Hence it suffices to estimate each $|\tr\rho_{\chi}(0,f_F)|$ separately.
For a $\cpt_{\infty}$-type $\tau\in\widehat{\cpt_{\infty}}$  let $\AAA_{\chi}^{\tau}$ denote the $\tau$-isotypic component of $\AAA_{\chi}$ so that $\AAA_{\chi}=\bigoplus_{\tau\in\widehat{\cpt_{\infty}}} \AAA_{\chi}^{\tau}$. Let $\AAA_{\chi}^{\tau,\cpt_f^F}$ denote the $\cpt_f^F$-fixed vectors in $\AAA_{\chi}^{\tau,\cpt_f^F}$. Then $\dim\AAA_{\chi}^{\tau, \cpt_f^F}<\infty$, and we can estimate
\[
|\tr\rho_{\chi}(0, f_F)|
 \leq \sum_{\tau\in\widehat{\cpt_{\infty}}} \dim\AAA_{\chi}^{\tau, \cpt_f^F} \|\rho(0,f_F)_{\AAA_{\chi}^{\tau,\cpt_f^F}}\|,
 \leq \sum_{\tau\in\widehat{\cpt_{\infty}}} \dim\AAA_{\chi}^{\tau, \cpt_f^F} \|f_F\|_{L^1(G(\A_F)^1)}
\]
where $\rho(0,f_F)_{\AAA_{\chi}^{\tau,\cpt_f^F}}$ denotes the restriction of $\rho(0,f_F)$ to $\AAA_{\chi}^{\tau,\cpt_f^F}$, and  $\|\rho(0,f_F)_{\AAA_{\chi}^{\tau,\cpt_f^F}}\|$ denotes the operator norm.
Now
\[
\|f_F\|_{L^1(G(\A_F))}
= \vol_{G(\A_{F,f})}(\cpt_f^F)\|f_{\infty}\|_{L^1(G(\R^{\signature})^1)}  ,
\]
and for fixed $f_{\infty}$ only finitely many $\cpt_{\infty}$-types can contribute to the above sum, and the set of contributing $\cpt_{\infty}$-types is independent of $F$. Hence it remains to show that we  can estimate the dimension of $\AAA_{\chi}^{\tau,\cpt_f^F}$ independently of $F$. As in the proof of \cite[Corollary 7.4]{FiLaMu_limitmult} we can compute
\begin{equation}\label{eq:dimension}
\dim\AAA_{\chi}^{\tau,\cpt_f^F}
= \dim \Ind_{B(\R^{\signature})}^{G(\R^{\signature})} (\chi_{\infty})^{\tau} \Ind_{B(\A_{F,f})}^{G(\A_{F,f})} (\chi_f)^{\cpt_f^F}
= \dim \Ind_{B(\R^{\signature})}^{G(\R^{\signature})} (\chi_{\infty})^{\tau}
\end{equation}
since $\chi$ is unramified. This last term does only depend on $\signature$ and $\tau$ but not on $F\in\F_{\signature}$. This completes the proof of the lemma.
\end{proof}

\begin{lemma}\label{lem:spec:term2}
There exists $c>0$ such that for all $F\in\F_{\signature}$ we have
\[
\int_{i\R} \left| \tr\left(M(\lambda)^{-1}M'(\lambda) \rho(\lambda, f_F)\right)\right| \,d\lambda
\leq
ch_F^{a_G}\log D_F \vol_{G(\A_{F,f})}( \cpt_f^F )
\]
with $a_G$ as in Lemma \ref{lem:spec:term1}.
\end{lemma}
\begin{proof}
As in the proof of Lemma \ref{lem:spec:term1} we decompose $\AAA$ according to the pairs of characters $\chi=(\chi_1,\chi_2)$, and get for every integer $k>0$ that (cf. \cite[\S 5]{FiLaMu}):
\begin{align*}
& \int_{i\R} \left| \tr\left(M(\lambda)^{-1}M'(\lambda) \rho(\lambda, f_F)\right)\right| \,d\lambda
 \leq \sum_{\chi} \int_{i\R} \left| \tr\left(M(\lambda)^{-1}M'(\lambda) \rho_{\chi}(\lambda, f_F)\right)\right| \,d\lambda \\
& \leq  C_k \vol_{G(\A_{F,f})}( \cpt_f^F ) \sum_{\chi, \tau} \dim\AAA_{\chi}^{\tau,\cpt_f^F}  \int_{i\R} \| M(\chi,\lambda)^{-1} M'(\chi,\lambda) _{|\AAA_{\chi}^{\tau,\cpt_f^F}} \| ~ (1+|\lambda|)^{-k}\,d\lambda,
\end{align*}
where $C_k>0$ is a constant depending only on $k$ and $f_{\infty}$, and the sum runs over $\chi$ and $\tau$ for which the restriction of $\rho(\lambda, f_F)$ to $\AAA_{\chi}^{\tau,\cpt_f^F}$ is non-zero. Here $M(\chi,\lambda)$ denotes the $M(\lambda)$ restricted to $\AAA_{\chi}$. Since $\chi$ is unramified, we have
\[
M(\chi,\lambda)^{-1}M'(\chi,\lambda)
= \frac{n'(\chi,\lambda)}{n(\chi,\lambda)}\id + R_{\infty}(\chi,\lambda)^{-1}R'_{\infty}(\chi,\lambda),
\]
where $n(\chi,\lambda)$ is a scalar normalization factor, and $R_{\infty}(\chi,\lambda)$ denotes the local normalized intertwining operator at $\infty$.
Again, as in the proof of Lemma \ref{lem:spec:term1} there are only $h_F^2$ (resp. $h_F$) many $\chi$'s which may contribute to the above sum if $G=\GL(2)$ (resp. $G=\SL(2)$), and the contributing $\cpt_{\infty}$-types depend only on $f_{\infty}$ and their number is finite.

By \cite[\S 5]{FiLaMu} there exist constants $a_k, b_k>0$ independent of $\tau\in\widehat{\cpt_{\infty}}$ and $F\in\F_{\signature}$ such that
\[
\int_{i\R} \| R(\chi,\lambda)^{-1} R'(\chi,\lambda) _{|\AAA_{\chi}^{\tau,\cpt_f^F}} \| ~(1+|\lambda|)^{-k}\,d\lambda
\leq a_k(1+\|\tau\|)^{b_k}.
\]
Taking into account dimension formula \eqref{eq:dimension} we are left to estimate
\[
\int_{i\R} \left|\frac{n'(\chi,\lambda)}{n(\chi,\lambda)}\right| ~(1+|\lambda|)^{-k}\,d\lambda.
\]
We have
\[
n(\chi,\lambda)
=\frac{L(1-\lambda,\tilde{\chi_1}\times\chi_2)}{L(1+\lambda,\chi_1\times\tilde{\chi_2})}
=\frac{L(1-\lambda, \chi_1^{-1}\chi_2)}{L(1+\lambda,\chi_1\chi_2^{-1})},
\]
where $L(s,\chi_1\times\chi_2)$ is the completed Rankin-Selberg $L$-function.
Hence if $\chi_1^{-1}\chi_2=1$, then
\[
n(\chi,\lambda)
=\frac{\zeta_F^*(1-\lambda)}{\zeta_F^*(1+\lambda)}
\]
for $\zeta_F^*(s)$ the completed Dedekind zeta function attached to $F$. Taking $k=4$ in above estimates and using Lemma \ref{lem:winding:number} below then finishes the proof of the assertion.
\end{proof}

\begin{lemma}\label{lem:winding:number}
There exist a constant $c>0$ depending only on $d$ such that for all $F\in\F_{\signature}$ and
all unramified unitary characters $\mu:F^{\times}\backslash \A_F^1\longrightarrow\C$ we have
\[
\int_{\R} \left|\frac{L'(1+it,\mu)}{L(1+it, \mu)}\right| ~(1+|t|)^{-4}\,dt
\leq c\log D_F.
\]
\end{lemma}

\begin{proof}
We follow the arguments in \cite[\S 4, \S 5]{Mu07}. (At some places we can do a little better in our setting.)
Recall that $L(s,\mu)=\prod_{v} L_v(s, \mu_v)$ with $v$ running over all places of $F$, and
\[
L_v(s,\mu_v)
=
\begin{cases}
(1-\mu(\qqq_v)q_v^{-s})^{-1}
&
\text{if }v\text{ is non-archimedean},\\
\Gamma_{\C}(s)
&
\text{if }v\text{ is complex},\\
\Gamma_{\R}(s)
&
\text{if }v\text{ is real and }\mu_v(-1)=1,\\
\Gamma_{\R}(s+1)
&
\text{if }v\text{ is real and }\mu_v(-1)=-1,
\end{cases}
\]
where
\[
\Gamma_v(s)
=\begin{cases}
\pi^{-s/2}\Gamma(s/2)
&\text{if }v=\R\\
2 (2\pi)^{-s}\Gamma(s)
&\text{if }v=\C
\end{cases}
\]
and $\Gamma(s)$ denotes the usual $\Gamma$-function.

Since $\mu$ is unitary, $|\mu(\qqq_v)|=1$ for any $v$ so that
\[
\left|\frac{L_f'(s,\mu)}{L_f(s,\mu)}\right|
\leq \left|\frac{\zeta_F'(\Re(s))}{\zeta_F(\Re(s))}\right|
\leq d \left|\frac{\zeta'(\Re(s))}{\zeta(\Re(s))}\right|
\]
for any $s$ with $\Re(s)>1$.

Using Stirling's formula as in the proof of \cite[Lemma 4.4]{Mu07} we can find for every $\delta>0$ a constant $a_{\delta}>0$ depending only on $\signature$ and $\delta$ such that
\[
\left|\frac{L_{\infty}'(s,\mu)}{L_{\infty}(s,\mu)}\right|
\leq a_{\delta}(1+\log|s|)
\]
for every $s$ with $\Re(s)\geq 1+\delta$. Hence for any $\delta>0$ there is $b_{\delta}>0$ independent of $F$ and $\mu$ such that
\begin{equation}\label{eq:bound:log:der:L}
\left|\frac{L'(s,\mu)}{L(s, \mu)}\right|
\leq b_{\delta}(1+\log|s|)
\end{equation}
for all $s$ with $\Re(s)\geq 1+\delta$.

Let $\nu(\mu)$ denote the level of $\mu$ (see \cite[(4.22)]{Mu07} for a definition). There exists an absolute constant $c_2>0$ depending only on $\signature$ such that
\begin{equation}\label{eq:bound:level}
0\leq \nu(\mu)
\leq c_1 D_F N(\mu)=c_1 D_F
\end{equation}
for all $F\in\F_{\signature}$ and all $\mu$. Here $N(\mu)$ denotes the absolute norm of the conductor of $\mu$ which in our situation is $1$.

Arguing as in the proof of  \cite[Proposition 4.5, Proposition 5.1]{Mu07} and using \eqref{eq:bound:log:der:L} and \eqref{eq:bound:level} we get a constant $C>0$  depending only on $\signature$ such that for all $T>1$ we have
\[
\int_{-T}^T \left|\frac{L'(1+it,\mu)}{L(1+it, \mu)}\right|\,dt
\leq C T (1+\log T + \log D_F).
\]
This implies our claim.
\end{proof}

\begin{proof}[Proof of Proposition \ref{prop:spectral:limit}]
The proposition now easily follows from Lemma \ref{lem:spec:term1} and Lemma \ref{lem:spec:term2} by bounding the measures of the involved groups similarly as for the geometric side. 
It follows from the above lemmas that  for all $F\in\F_{\signature}$ we have
\begin{multline*}
\left|\frac{J_{\text{spec}}^F(f_{\infty}\cdot\One_{\cpt_f^F})
-J_{\text{disc}}^F(f_{\infty}\cdot\One_{\cpt_f^F})}{\vol(G(F)\backslash G(\A_F)^1)}\right|
\ll_{\signature} \frac{1}{\nu_F} D_F^{a_G/2} (\log D_F)^{a_G(d-1)} \\
\ll_{\signature, \eps} 
\begin{cases}
  \nu_F^{-\frac{1}{2}+\eps}			&\text{if }G=\GL(2),\\
 \nu_F^{-\frac{2}{3}+\eps}			&\text{if } G=\SL(2)
\end{cases}
\end{multline*}
for any $\eps>0$. This tends to $0$ as $F$ varies over $\F_{\signature}$.
\end{proof}

 \newcommand{\etalchar}[1]{$^{#1}$}
\providecommand{\bysame}{\leavevmode\hbox to3em{\hrulefill}\thinspace}
\providecommand{\MR}{\relax\ifhmode\unskip\space\fi MR }
\providecommand{\MRhref}[2]{%
  \href{http://www.ams.org/mathscinet-getitem?mr=#1}{#2}
}
\providecommand{\href}[2]{#2}


\end{document}